\documentclass[a4paper, 12pt]{article}%

\usepackage[utf8]{inputenc}
\usepackage{amsmath}
\usepackage{amsfonts}
\usepackage{amssymb}
\usepackage{amsthm}
\usepackage{graphicx}
\usepackage{subcaption}
\usepackage{epstopdf}
\usepackage[export]{adjustbox}
\usepackage{multirow}
\usepackage{booktabs}
\usepackage{color}

\newtheorem{theorem}{Theorem}
\newtheorem{lemma}[theorem]{Lemma}
\newtheorem{definition}[theorem]{Definition}
\newtheorem{remark}[theorem]{Remark}

\newcommand{\diam}{\operatorname*{diam}}

\newcommand{\vertiii}[1]{{\left\vert\kern-0.25ex\left\vert\kern-0.25ex\left\vert #1 
    \right\vert\kern-0.25ex\right\vert\kern-0.25ex\right\vert}}

\newcommand{\tri}{\mathcal{T}}

\newcommand{\triH}{\tri_H}
\newcommand{\trih}{\tri_h}

\newcommand{\umsH}{\vec u^{\operatorname*{ms}}_{H,h}}
\newcommand{\tildeumsH}{\tilde{\vec u}^{\operatorname*{ms},k}_{H,h}}
\newcommand{\umsHk}{\vec u^{\operatorname*{ms},k}_{H,h}}
\newcommand{\umsHkell}{\vec u^{\operatorname*{ms},k,\ell}_{H,h}}

\newcommand{\vmsH}{\vec v^{\operatorname*{ms}}_{H,h}}
\newcommand{\pH}{p_{H}}

\newcommand{\VmsH}{V^{\operatorname*{ms}}_{H,h}}
\newcommand{\VmsHk}{V^{\operatorname*{ms},k}_{H,h}}
\newcommand{\KmsH}{K_{H,h}^{\operatorname*{ms}}}

\newcommand{\norm}[1]{\left\lVert#1\right\rVert}
\newcommand{\R}{\mathbb{R}}
\newcommand{\T}{\mathcal{T}}
\newcommand{\RT}{\mathcal{RT}}

\newcommand{\IH}{\Pi_H}
\newcommand{\Ih}{\Pi_h}

\newcommand{\Qh}{Q_h}

\newcommand{\Qhf}{Q_h^{\operatorname*{f}}}
\newcommand{\oQhf}[1]{Q_h^{\operatorname*{f}}(#1)}

\newcommand{\Vh}{V_h}

\newcommand{\Vhf}{V_h^{\operatorname*{f}}}

\newcommand{\oVhf}[1]{V_h^{\operatorname*{f}}(#1)}

\newcommand{\vf}{v^{\operatorname*{f}}}

\newcommand{\Kh}{K_h}

\newcommand{\Khf}{K_h^{\operatorname*{f}}}
\newcommand{\oKhf}[1]{K_h^{\operatorname*{f}}(#1)}

\newcommand{\Hdiv}{H(\mathrm{div}, \Omega)}
\newcommand{\Hzdiv}{H_0(\mathrm{div}, \Omega)}
\newcommand{\HdivRd}{H(\mathrm{div}, \R^d)}

\newcommand{\oPoly}[2]{\mathbb{P}^{#1}(#2)}

\newcommand{\Ltwo}{L^2(\Omega)}

\renewcommand{\vec}[1]{\mathbf{#1}}

\newtheorem{assumption}{Assumption}

\usepackage{geometry}
\begin{document}

\begin{center}
 {\LARGE Multiscale mixed finite elements}\\[2em]
\end{center}

\renewcommand{\thefootnote}{\fnsymbol{footnote}}
\renewcommand{\thefootnote}{\arabic{footnote}}

\begin{center}
{
 \large Fredrik Hellman\footnote[1]{Department of Information Technology, Uppsala University, Box 337, SE-751 05 Uppsala, Sweden. Supported by Centre for Interdisciplinary Mathematics (CIM), Uppsala University.},
  Patrick Henning\footnote[2]{Department of Mathematics, KTH Royal Institute of Technology, SE-100 44 Stockholm, Sweden.},
  Axel M\r{a}lqvist\footnote[3]{Department of Mathematical Sciences,
Chalmers University of Technology and
University of Gothenburg 
SE-412 96 G\"oteborg, Sweden. Supported by the Swedish Research Council.}
}
\end{center}

\begin{center}
{\large{\today}}
\end{center}

\begin{center}
\end{center}

\begin{abstract}
In this work, we propose a mixed finite element method for solving elliptic multiscale problems based on a localized orthogonal decomposition (LOD) of Raviart--Thomas finite element spaces. It requires to solve local problems in small patches around the elements of a coarse grid. These computations can be perfectly parallelized and are cheap to perform. Using the results of these patch problems, we construct a low dimensional multiscale mixed finite element space with very high approximation properties. This space can be used for solving the original saddle point problem in an efficient way. We prove convergence of our approach, independent of structural assumptions or scale separation. Finally, we demonstrate the applicability of our method by presenting a variety of numerical experiments, including a comparison with an MsFEM approach.
\end{abstract}



\section{Introduction}
In this work we study the mixed formulation of Poisson's equation with
a multiscale diffusion coefficient, i.e.\ where the diffusion
coefficient is highly varying on a continuum of different scales. 
For such coefficients, the solution is typically also
highly varying and standard Galerkin methods fail to converge to the correct solution,
unless the features on the finest scale are resolved by the underlying computational mesh.
A classical application is the flow in a porous medium, modeled by Darcy's law.
In this case, the multiscale coefficient describes
a permeability field, which is heterogeneous, rapidly varying and has high contrast. 
Classical discretizations that involve the full fine scale often lead
to a vast number of degrees of freedom, which limits the performance
and feasibility of corresponding computations.  In this paper, we
address this kind of problems in the context of mixed finite
elements.

We will interpret the mixed formulation of Poisson's equation in a
Darcy flow setting, referring to the vector component as flux, and the
scalar component as pressure. In Darcy flow applications the flux
solution is of particular interest since it tells us how a fluid is
transported through the medium. It is desirable and common to use flux
conservative discretization schemes. The proposed method is based on
the Raviart--Thomas finite element \cite{RaTh77} which is locally flux
conservative.
Concerning the mixed formulation of Poisson's equation,
corresponding multiscale methods were for instance proposed in \cite{Aa04, Ar11, ArBo06, ChHo02}.
These methods are based on the Raviart--Thomas finite element and
fit into the framework of the Multiscale Finite Element Method (MsFEM, cf. \cite{HoWu97}).
Another family of multiscale methods is derived
from the framework of the Variational Multiscale Method (VMS)
\cite{HuSa07,Hu95,HuFeGoMaQu98,LaMa07,NoPaPi08}. Multiscale methods
for mixed finite elements based on VMS are proposed and studied in
\cite{Arb04,LaMa09,Ma11}. Inspired by the results presented in \cite{Ma11},
a new multiscale framework arose \cite{MaPe14}.  We refer to this
framework as Localized Orthogonal Decomposition (LOD). It is based on
the idea that a finite element space is decomposed into a low
dimensional space that incorporates multiscale features and a high
dimensional remainder space which is given as the kernel of an
interpolation or quasi-interpolation operator. The multiscale space
can be used for Galerkin-approximations and allows for cheap
computations.  Various realizations have been proposed so far.  For
corresponding formulations and rigorous convergence results for
elliptic multiscale problems, we refer to
\cite{AbH15,ElGiHe14,HeMa14,HeP13,MaPe14} for Galerkin finite element
methods, to \cite{ElGeMa13,ElGiHe14} for discontinuous Galerkin
methods and to \cite{HMP13b} for Galerkin Partition of Unity methods.
Among the various applications we refer to the
realizations for eigenvalue problems \cite{MaP14b}, for semilinear
equations \cite{HeMaPe14}, for the wave equation \cite{AbH14c} and for
the Helmholtz equation \cite{Pet14}.

In this paper we introduce a two level discretization of the mixed
problem, that is we work with two meshes: A fine mesh (mesh size $h$)
which resolves all
the fine scale features in the solution and a coarse mesh (mesh size
$H$) which is of computationally feasible size. This gives us a fine
and a coarse Raviart--Thomas function space for the flux. We denote them 
respectively by
$V_h$ (high dimensional) and $V_H$ (low dimensional).
The kernel of
the (standard) nodal Raviart--Thomas interpolation operator $\IH$ onto
$V_H$ is the detail space $\Vhf$. This space can be interpreted as all
fine scale features that can not be captured in the coarse space
$V_H$. A low dimensional ideal multiscale space is constructed as the
orthogonal complement to the divergence free fluxes in $\Vhf$. We
prove that this space has good approximation properties in the sense
that the energy norm of the error converges with $H$ without
pre-asymptotic effects due to the multiscale features.
However, the basis functions of the ideal
multiscale space have global support and are expensive to compute. We
show exponential decay of these basis functions allowing them to be
truncated to localized patches with a preserved order of accuracy for
the convergence. The resulting space is called the localized
multiscale space. The problems 
that are associated with the
localized basis functions have a small
number of degrees of freedom and can be solved in parallel with
reduced computational cost and memory requirement. Once computed, the
low dimensional localized multiscale space can be reused in a
nonlinear or time iterative scheme.

We prove inf-sup stability and a priori error estimates (of linear
order in $H$) for both the ideal and the localized method. The local
$L^2$-instability of the nodal Raviart--Thomas interpolation operator
leads to instabilities as $h$ decreases for the localized method. We
show that these instabilities can be compensated by increasing the
patch size or using Cl\'ement-type interpolators instead. In the
numerical examples we verify that the localized method has the
theoretically derived order of accuracy. We confirm our theoretical
findings by performing experiments on the unit square and an L-shaped
domain, as well as using a diffusion coefficient with high contrast
noise and channel structures. The proposed method is also compared
numerically with results from an MsFEM-based approach using a
permeability field from the SPE10 benchmark problem.

\section{Preliminaries}
We consider a bounded Lipschitz domain $\Omega\subset \R^d$ (dimension $d = 2$ or $3$)
with a piecewise polygonal boundary
$\partial \Omega$ and let $\vec n$ denote the outgoing normal vector
of $\partial \Omega$. For any subdomain
$\omega \subseteq \Omega$, we shall use standard notation for Lebesgue and Sobolev spaces,
i.e.\ for $r\in[1,\infty]$, $L^r(\omega)$ consists of measurable functions with bounded $L^r$-norm and
the space $H^1(\omega)$ consists of $L^2$-bounded weakly differentiable functions
with $L^2$-bounded partial derivatives. The full norm on $H^1(\omega)$
shall be denoted by $\| \cdot \|_{H^1(\omega)}$, whereas
the semi-norm is denoted by $| \cdot |_{H^1(\omega)}:=\| \nabla \cdot \|_{L^2(\omega)}$.

For scalar functions $p$ and $q$ we denote by $(p, q)_{\omega} := \int_\omega
p\,q$ the $L^2$-scalar product on $\omega$. When $\omega = \Omega$, we
omit the subscript, i.e.\ $(p, q) := (p, q)_{\Omega}$. For
$d$-dimensional vector valued functions $\vec u$ and $\vec v$, we
define $(\vec u, \vec v)_{\omega} := \int_\omega \vec u \cdot \vec v $
with $(\vec u, \vec v) = (\vec u, \vec v)_{\Omega}$. 
Observe that we use the same notation for norms and scalar products in
$L^2$ without distinguishing between scalar and vector valued functions.
This is purely for simplicity, since the appropriate definition is always clear from the context.
We use, however, bold face letters for vector valued quantities. 

In the following, we
define the Sobolev space of
functions with $L^2$-bounded weak divergence by
$H(\mathrm{div},\omega) := \{ \vec v \in
[L^2(\omega)]^d : \nabla \cdot \vec v \in L^2(\omega) \}$. We equip this
space with the usual
norm $\| \cdot \|_{H(\mathrm{div}, \omega)}$, where $\| \vec v \|_{H(\mathrm{div}, \omega)}^2 := \| \nabla \cdot \vec v \|_{L^2(\omega)}^2 + \|\vec v\|_{L^2(\omega)}^2$. 
Additionally, for $\omega=\Omega$, we introduce the
subspace $\Hzdiv := \{ \vec v \in H(\mathrm{div},\Omega) : \vec v
\cdot \vec n|_{\partial \Omega} = 0 \}$ of functions with zero flux on
the boundary, where $\vec v \cdot \vec n|_{\partial \Omega}$ should be
interpreted in the sense of traces. We denote by $L^2(\Omega)/\R := \{
q \in L^2(\Omega) : \int_\Omega q = 0\}$ the quotient space of
$L^2(\Omega)$ by $\R$. The continuous dual space of a Banach space $X$
is denoted by $X'$.

\subsection{Continuous problem}
With these definitions we are ready to state the continuous problem,
which is Poisson's equation in mixed form with Neumann boundary
conditions on the full boundary.

\begin{definition}[Continuous problem]
  Find $\vec{u} \in V := \Hzdiv$, $p \in Q := \Ltwo/\R$ such that
\begin{equation}
  \label{eq:mixed}
  \begin{aligned}
    \left({\vec A}^{-1} \vec u, \vec v \right) + (\nabla \cdot \vec v, p) & = 0, \\
    \left(\nabla \cdot \vec u, q\right) & = -(f, q),
  \end{aligned}
\end{equation}
for all $\vec v \in V$, $q \in Q$.
\end{definition}
We pose the following assumptions on the coefficient and data.
\begin{assumption}[Assumptions on coefficients, data and domain]$\\$\vspace{-15pt}
  \label{ass:data}
  \begin{enumerate}
  \item[{\upshape (A1)}] ${\vec A} \in [L^{\infty}(\Omega)]^{d \times d}$ is
    a diffusion coefficient, possibly with rapid fine scale
    variations. Its value is an almost everywhere symmetric matrix and
    bounded in the sense that there exist real numbers $\alpha$ and
    $\beta$ such that for almost every $x$ and any $\vec v \in
    \R^d/\{0\}$
    \begin{equation*}
      0 < \alpha \le \frac{({\vec A}(x)^{-1}\vec v)\cdot\vec v}{\vec v\cdot\vec v} \le \beta < \infty.
    \end{equation*}
  \item[{\upshape (A2)}] $f \in L^2(\Omega)$ is a source function that
    fulfills the compatibility condition $\int_\Omega f = 0$.
  \item[{\upshape (A3)}] The domain $\Omega$ is a bounded Lipschitz domain with
    polygonal (or polyhedral) boundary.
  \end{enumerate}
\end{assumption}
We introduce the following bilinear forms and norms. Let
\begin{equation*}
  a(\vec u, \vec v) :=
  ({\vec A}^{-1} \vec u, \vec v)\quad \text{ and } \quad b(\vec v, q) := (\nabla \cdot \vec v,
  q)
\end{equation*}
and, further,
\begin{equation*}
\| \vec v \|_V := \| \vec v \|_{\Hdiv}\quad \text{ and } \quad\| q \|_Q := \| q \|_{L^2(\Omega)}.
\end{equation*}
The energy norm is defined as the following weighted flux $L^2$-norm,
\begin{equation*}
\vertiii{\vec v}^2 := \| {\vec A}^{-1/2} \vec v \|_{L^2(\Omega)}^2 = a(\vec v, \vec v)
\end{equation*}
The energy norm can be subscripted with a subdomain $\omega \subseteq
\Omega$, for example $\vertiii{\cdot}^2_\omega$, to indicate that the
integral is taken only over that subdomain.

The following lemma gives the conditions for existence and uniqueness
of a solution to the mixed formulation in \eqref{eq:mixed} for
subspaces $\mathcal{V} \subseteq V$ and $\mathcal{Q} \subseteq
Q$. This lemma is useful for establishing existence and uniqueness for
all discretizations presented in this paper, since all presented
discretizations are conforming.

\begin{lemma}[Existence and uniqueness of solution to mixed formulation]
  \label{lem:existence}
  Let $\mathcal{V} \subseteq V$ and $\mathcal{Q} \subseteq Q$. Denote
  by $\mathcal{K} = \{\vec v \in \mathcal{V} : b(\vec v, q) = 0 \ 
  \forall q \in \mathcal{Q}\}$. If $a(\cdot,\cdot)$ is coercive on $\mathcal{K}$
  with constant $\tilde \alpha > 0$, i.e.\ $a(\vec v, \vec v) \ge \tilde \alpha \|
  \vec v \|_{V}^2$ for $\vec v \in \mathcal{K}$, and bounded with
  constant $\tilde \beta > 0$, i.e. $|a(\vec v, \vec w)| \le \tilde \beta \| \vec
  v\|_V \| \vec w\|_V$ for all $\vec v, \vec w \in \mathcal{V}$, and
  additionally $b(\cdot,\cdot)$ is inf-sup stable with constant $\tilde \gamma > 0$,
  i.e.\
  \begin{equation*} 
    \inf_{q \in \mathcal{Q}} \sup_{\vec v \in \mathcal{V}} \frac{b(\vec v, q)}{\| \vec v \|_V \| q \|_Q} \ge \tilde \gamma,
  \end{equation*}
  then the problem $a(\vec u, \vec v) + b(\vec v, p) - b(\vec u, q) =
  (f, q)$ for all $(\vec v, q) \in \mathcal{V} \times \mathcal{Q}$ has
  a unique solution $(\vec u, p) \in \mathcal{V} \times
  \mathcal{Q}$ bounded by
  \begin{equation*}
      \| \vec u \|_V \le \frac{2\tilde \beta^{1/2}}{\tilde \alpha^{1/2}\tilde \gamma}\|f\|_{L^2(\Omega)} \quad\text{and}\quad \| p \|_Q \le \frac{\tilde \beta}{\tilde \gamma^2}\|f\|_{L^2(\Omega)}.
  \end{equation*}
\end{lemma}
\begin{proof}
  See e.g.\ \cite[Theorem 4.2.3]{BoBrFo13}.
\end{proof}
Under Assumptions~(A1)--(A3), the conditions for
Lemma~\ref{lem:existence} are satisfied for $\mathcal{V} = V$ and
$\mathcal{Q} = Q$ with $\tilde \alpha = \alpha$, $\tilde \beta =
\beta$ and $\tilde \gamma$ being a constant that depends only on the
computational domain. The lemma then yields a unique solution to the continuous
problem \eqref{eq:mixed}.

\subsection{Discretization with the Raviart--Thomas element}

Regarding the discretization, we introduce two conforming families of
simplicial (i.e.\ triangular or tetrahedral) meshes $\{\trih\}$ and $\{\triH\}$ of $\Omega$ where $h$ and $H$ are
the maximum element diameters. Throughout the paper 
we refer to
$\trih$ as the fine mesh and 
to $\triH$ as the coarse mesh.
Hence, we indirectly assume $h \ll H$. We pose the following assumptions on the
meshes.
\begin{assumption}[Assumptions on meshes]$\\$\vspace{-15pt}
  \label{ass:mesh}
\begin{enumerate}
\item[{\upshape (B1)}] The fine mesh $\trih$ is the result of one or more
  conforming (but possibly non-uniform) refinements of the coarse mesh
  $\triH$ such that $\trih \cap \triH = \emptyset$.
\item[{\upshape (B2)}] Both meshes $\trih$ and $\triH$ are shape regular. In particular the
  positive shape regularity constant $\rho$ for the coarse mesh
  $\triH$ will be referred to below and is defined as $\rho =
  \min_{T\in\triH} \frac{\diam B_T}{\diam T}$ where $B_T$ is the
  largest ball contained in the element $T\in\triH$.
\item[{\upshape (B3)}] The coarse family of meshes $\{\triH\}$ is
  quasi-uniform, whereas $\{\trih\}$ could be obtained from an
  arbitrary adaptive refinement.
\end{enumerate}
\end{assumption}
\begin{remark}[Quadrilateral or hexahedral elements]
  Affine quadrilateral (or hexahedral) elements can also be
  used. However, the definition of the Raviart--Thomas element
  presented below in this paper is based on triangular (or
  tetrahedral) meshes.
\end{remark}
We denote by $t$ and $T$ an element of $\trih$ or $\triH$,
respectively. Similarly $e$ and $E$ denote an edge (for $d=2$) or a face (for $d=3$) of
the elements of $\trih$ and $\triH$. Further, $\vec n_e$ (respectively $\vec n_E$) is
the outward normal vector of an edge (or face) $e$ (respectively $E$). We continue this
section by discussing finite element discretizations using the two meshes.

We denote all polynomials of degree $\le k$ on a subdomain $\omega$ by
$\oPoly{k}{\omega}$ and a $d$-dimensional vector of such polynomials
by $[\oPoly{k}{\omega}]^d$. We introduce the $\Hzdiv$-conform\-ing
lowest (zeroth) order Raviart--Thomas finite element. For each fine
element $t \in \trih$ and coarse element $T \in \triH$, the spaces of
Raviart--Thomas shape functions are given by
\begin{equation*}
  \begin{aligned}
  \RT_{h}(t) &= \{\vec v|_t = [\oPoly{0}{t}]^d + x \oPoly{0}{t}\} \text{ and} \\
  \RT_{H}(T) &= \{\vec v|_T = [\oPoly{0}{T}]^d + x \oPoly{0}{T}\},
  \end{aligned}
\end{equation*}
respectively, where $x = (x_1, \ldots, x_d)$ is the space
coordinate vector. The Raviart--Thomas finite element spaces on
$\trih$ and $\triH$ are then defined as
\begin{equation*}
  \begin{aligned}
  \Vh &= \{\vec v \in \Hzdiv : \vec v|_t \in \RT_h(t) \quad\forall t\in\trih\} \text{ and} \\
  V_H &= \{\vec v \in \Hzdiv : \vec v|_T \in \RT_H(T) \quad\forall T\in\triH\}.
  \end{aligned}
\end{equation*}
The degrees of freedom (in the coarse and fine Raviart--Thomas spaces)
are given by the averages of the
normal fluxes over the edges (respectively faces for $d=3$).
We denote the degrees of freedom by
\begin{equation*}
  \begin{aligned}
    N_e(\vec v) := \frac{1}{|e|} \int_e \vec v \cdot \vec n_e \quad \text{and} \quad N_E(\vec v) := \frac{1}{|E|} \int_E \vec v \cdot \vec n_E
  \end{aligned}
\end{equation*}
for the fine and coarse discretization, respectively. 
The direction of the normal $\vec n_e$ (respectively $\vec n_E$)
can be fixed arbitrarily for each edge (respectively face). Here, $N_e$ and $N_E$ are
bounded linear functionals on the space $W := \Hzdiv \cap
L^s(\Omega)$, for some $s > 2$. Note, that the additional regularity (i.e.\ 
$L^s(\Omega)$ for $s>2$) is necessary for the edge integrals to be
well-defined (cf. \cite{BoBrFo13}).
We introduce the (standard) nodal Raviart--Thomas interpolation
operators $\Ih : W \rightarrow \Vh$ and $\IH : W \rightarrow V_H$ by
fixing the degrees of freedom in the natural way, i.e.\ $\Ih$ and $\IH$
are defined such that
\begin{equation*}
  \begin{aligned}
    N_e(\Ih \vec v) = N_e(\vec v) \quad \text{and} \quad N_E(\IH \vec v) = N_E(\vec v).
  \end{aligned}
\end{equation*}
Additionally, we let $Q_H \subset \Qh \subset Q$ be the space of all
piecewise constant functions on $\triH$ and $\trih$ with zero
mean. We denote by $P_h$ and $P_H$ the $L^2$-projections onto $Q_h$
and $Q_H$, respectively. Using the fine spaces, we define the fine scale
discretization of \eqref{eq:mixed}, which will be referred to as the
reference problem.

\begin{definition}[Reference problem]
Find $\vec u_h \in \Vh$ and $p_h \in \Qh$, such that
\begin{equation}
  \label{eq:mixeddiscrete}
  \begin{aligned}
    a(\vec u_h, \vec v_h) + b(\vec v_h, p_h) & = 0, \\
    b(\vec u_h, q_h) & = -(f, q_h),
  \end{aligned}
\end{equation}
for all $\vec v_h \in \Vh$ and $q_h \in \Qh$.
\end{definition}

A similar problem can be stated with the coarse spaces $V_H$ and $Q_H$
with flux solution $\vec u_H$. The remainder of this section treats
only the fine discretization. However, all results hold also for the
coarse discretization.

We denote the space of divergence free functions on the fine grid by
\begin{equation}
  \label{eq:kh}
  \Kh := \{\vec v \in \Vh : \nabla \cdot \vec v = 0\}.
\end{equation}
\begin{remark}[Kernel of divergence operator]
  \label{rem:kernel}
  A natural definition of $\Kh$ for our purposes is $\Kh = \{
  \vec v \in \Vh : (\nabla \cdot \vec v, q_h) = 0\ \forall q_h \in Q_h
  \}$. However, since we have $\nabla \cdot \vec v \in Q_h$ for all
  $\vec v \in \Vh$ (due to the definition of the Raviart--Thomas
  element), we can characterize $\Kh$ equivalently as done in
  \eqref{eq:kh}.
\end{remark}

To establish existence and uniqueness of a solution to the reference
problem, we use that $\Pi_h$ is divergence compatible, i.e.\ we have
the commuting property $\nabla \cdot \Pi_h \vec v = P_h \nabla \cdot
\vec v$ for $\vec v \in W$, and that $\Pi_h$ is bounded on $W$ (but
not on $V$!), i.e.\ there exists a generic $h$-independent constant
$C_W$ such that $\| \Pi_h \vec v \|_V \le C_W \| \vec v \|_W$ for
$\vec v \in W$. Using this, the inf-sup stability of $b(\cdot,\cdot)$
with respect to $V_h$ and $Q_h$ follows: For $q \in Q_h$,
\begin{equation}
  \begin{aligned}
    \sup_{\vec v \in V_h} \frac{b(\vec v, q)}{\| \vec v \|_V } & = 
    \sup_{\vec v \in W} \frac{(\nabla \cdot \Pi_h \vec v, q)}{\| \Pi_h \vec v \|_V } \ge
    \sup_{\vec v \in W} \frac{(\nabla \cdot \vec v, q)}{C_W\| \vec v \|_W } \\
    &\ge \frac{(\nabla \cdot \vec w, q)}{C_W\| \vec w \|_W } \ge
    \frac{(q, q)}{C_WC_{\Omega}\| q \|_{L^2(\Omega)} } = C_W^{-1}C_{\Omega}^{-1} \| q \|_{L^2(\Omega)},\\
  \end{aligned}
\end{equation}
where $\vec w \in W$ is chosen such that $\nabla \cdot \vec w = q$ and
$\| \vec w \|_W \le C_{\Omega} \| q \|_{L^2(\Omega)}$. This is
possible by letting $\vec w = \nabla \phi$ for a solution $\phi$ to
$\Delta \phi = q$ in $\Omega$ with homogeneous Neumann boundary
conditions. Now, applying Lemma~\ref{lem:existence} with $\mathcal{V}
= V_h$, $\mathcal{Q} = Q_h$, $\mathcal{K} = K_h$, we can derive the
constants $\tilde \alpha = \alpha$, $\tilde \beta = \beta$ and $\tilde
\gamma = \gamma := C_W^{-1}C_{\Omega}^{-1}$ and establish existence
and uniqueness of a solution to the reference problem
\eqref{eq:mixeddiscrete}. Note that the inf-sup stability constant
$\gamma$ is independent of $h$ and hence also holds for the pair of
spaces $V_H$ and $Q_H$.

In the following, we are mainly interested in approximating the flux
component $\vec u_h$ of the solution. We treat $\vec u_h$ as a
reliable reference to the exact solution.  Note that the $L^2$-norm of
the divergence error is controlled by the data
\begin{equation*}
  \begin{aligned}
    \| \nabla \cdot \vec u - \nabla \cdot \vec u_h \|_{\Ltwo} & \le  \| f - P_h f \|_{\Ltwo}. \\
  \end{aligned}
\end{equation*}
For the energy norm of the flux error, we have the following error
estimate in the energy norm for the lowest order Raviart--Thomas element:
\begin{equation*}
  \begin{aligned}
    \vertiii{\vec u - \vec u_h} & \le C h | \vec u |_{H^{1}(\Omega)},
  \end{aligned}
\end{equation*}
where $C$ is independent of $h$.
For a problem with a coefficient ${\vec A}$
that has
fast variations
at a scale of size $\epsilon$, we have in
general that $| \vec u |_{H^{1}(\Omega)} \approx \epsilon^{-1}$. 
Hence, we require $h\ll \epsilon$ before we can observe the linear convergence
in $h$ numerically. We call the regime with $h\ge \epsilon$ a pre-asymptotic
regime. The goal of this work is the construction of a discrete space
which does not suffer from such pre-asymptotic effects triggered by ${\vec A}$.
In the following, we 
assume that the fine mesh is fine enough (in the sense that $h \ll \epsilon$) so that
$\vertiii{\vec u - \vec u_h}$ is sufficiently small and hence $\vec u_h$ a sufficiently accurate
reference solution. With the same
argument, the accuracy of the coarse solution $\vec u_H$ will not be
satisfying as long as $H > \epsilon$. Note that reference 
problem \eqref{eq:mixeddiscrete} never needs to be solved. It just serves as
a reference. 

In the next section, we will construct the ideal multiscale space of
the same (low) dimension as $V_H$, but which yields approximations
that are of similar accuracy as the reference solution $\vec u_h$ (in
particular in the regime $H \gg \epsilon$). Throughout the paper, we
do not consider errors that arise from numerical quadrature.  For
simplicity, we assume that all integrals can be computed exactly.

\section{Ideal multiscale problem}

In this section, we construct a low dimensional space that can capture
the fine scale features of the true multiscale solution. We focus on
constructing a good multiscale representation of the flux solution
$\vec u$ only. We call it ideal since the reference flux solution is
in this space for all $f \in Q_H$. This should be contrasted to a
localized multiscale space to be introduced in
Section~\ref{sec:localized}. In addition to the spaces $V_h$ and $V_H$
defined above we introduce the following detail space as the
intersection of the fine space and the kernel of the coarse
Raviart--Thomas interpolation operator,
\begin{equation*}
  \begin{aligned}
    \Vhf & = \{\vec v \in \Vh : \Pi_H \vec v = 0 \}. \\
  \end{aligned}
\end{equation*}
Since $\Vhf$ is the kernel of a projection, it induces the splitting
$\Vh = V_H \oplus \Vhf$, where $V_H$ is low dimensional and $\Vhf$ is
high dimensional. We refer to $\Vhf$ as the detail space. 
In this section we aim at constructing a 
modified splitting, where $V_H$ is replaced by a multiscale
space which incorporates fine scale features.

\subsection{Ideal multiscale space}
We will construct the ideal multiscale space by applying fine scale
correctors to all coarse functions in $V_H$, i.e.\ so that 
$(\mbox{Id} - G_h)(V_H)$
is the desired multiscale space for a linear corrector operator
$G_h$. The corrector operator is constructed using information from
the coefficient ${\vec A}$, and has divergence free range in order to keep the
flux conservation property of the coarse space.

The definition of the corrector requires us to construct the splitting
$K_h = K_H \oplus \Khf$ with
\begin{equation*}
  \begin{aligned}
    \Khf &:= \{ \vec v \in K_h : \Pi_H \vec v = 0 \}, \quad 
    \mbox{and} \quad
    K_H := \mbox{Range}((\Pi_H)_{\vert K_h}).
  \end{aligned}
\end{equation*}
Next, we introduce an ideal corrector operator. We distinguish
between local (element-wise) correctors and a global corrector.
\begin{definition}[Ideal corrector operators]
  \label{def:idealcorrector}
 Let $a^T(\vec u, \vec v) := ({\vec A}^{-1} \vec u, \vec v)_T$ for $T \in \triH$.
 For each such $T \in \triH$, we define an \emph{ideal element corrector
    operator} $G^T_h : V \to \Khf$ by the equation
  \begin{equation}
    \label{eq:idealcorrector}
    a(G^T_h\vec v, \vec \vf) = a^T(\vec v, \vec \vf)
  \end{equation}
  for all $\vec \vf \in \Khf$. Furthermore, we define the
  \emph{ideal global corrector operator} by summing the local contributions, i.e.\ $G_h := \sum_{T\in\triH} G^T_h$.
\end{definition}
The ideal corrector operators are well-defined since equation
\eqref{eq:idealcorrector} is guaranteed a unique solution by the
Lax--Milgram theorem due to the coercivity and boundedness of
$a(\cdot,\cdot)$ on $\Khf$. Using the ideal global corrector operator, we can define the
discrete multiscale function space by
\begin{equation*}
  \label{eq:vmsh-definition}
  \VmsH := (\mbox{Id}-G_h)(V_H),
\end{equation*}
where $\mbox{Id}$ is the identity operator. This space has the same
dimension as $V_H$. Furthermore, it allows for the splitting $\Vh =
\VmsH \oplus \Vhf$. Note that the ideal multiscale space is the orthogonal
complement of $\Khf$ with respect to $a(\cdot,\cdot)$, i.e.\
\begin{equation}
  \label{eq:idealorthogonality}
  a(\vmsH, \vec \vf) = 0
\end{equation}
for all $\vmsH \in \VmsH$ and $\vec \vf \in \Khf$.

\subsection{Ideal multiscale problem formulation}

In this section, we use the previously defined ideal multiscale space
to define a (preliminary) multiscale approximation.  The ideal
multiscale problem reads as follows.

\begin{definition}[Ideal multiscale problem]
  Find $\umsH \in \VmsH$ and $\pH \in Q_H$, such that
  \begin{equation}
    \label{eq:msmixedbilinear}
    \begin{aligned}
      a(\umsH, \vec v_h) + b(\vec v_h, \pH) & = 0, \\
      b(\umsH, q_H) & = -(f, q_H),
    \end{aligned}
  \end{equation}
  for all $\vec v_h \in \VmsH$ and $q_H \in Q_H$.
\end{definition}

\begin{lemma}[Unique solution of the ideal multiscale problem]
  \label{lem:uniqueideal}
  Under Assumptions~(A1)--(A3) and (B1)--(B3), the ideal
  multiscale problem \eqref{eq:msmixedbilinear} has a unique
  solution. In particular, we have 
  \begin{align*}
    \gamma(1 + \alpha^{-1}\beta)^{-1} & \le \inf_{q \in Q_H}\sup_{\vec v \in \VmsH} \frac{b(\vec v, q)}{\norm{q}_{Q} \norm{\vec v}_V},
  \end{align*}
  i.e.\ inf-sup stability independent of $h$ and $H$.
\end{lemma}
\begin{proof} 
  We let $\KmsH = \{ \vec v \in \VmsH : \nabla \cdot \vec v =
  0\}$. The coercivity of $a(\cdot,\cdot)$ on $\KmsH$ follows immediately from its
  coercivity on $K_h$ since $\KmsH \subset K_h$.
  The operator $\mbox{Id} - G_h$ is stable in $V$ with
  constant $1 + \alpha^{-1}\beta$, since $\nabla \cdot G_h \vec v = 0$
  and
  \begin{equation*}
    \begin{aligned}
      \norm{G_h\vec v}^2_{\Ltwo} & \le \alpha^{-1} a(G_h\vec v, G_h\vec v) \\
      & = \alpha^{-1} a(\vec v, G_h\vec v) \\
      & \le \alpha^{-1}\beta \norm{\vec v}_{\Ltwo}\norm{G_h \vec v}_{\Ltwo}
    \end{aligned}
  \end{equation*}
  for all $\vec v \in V$. Combining these results with the inf-sup
  stability of $b(\cdot,\cdot)$ on $V_H$ and $Q_H$, we get
\begin{equation}
  \label{eq:infsup}
  \begin{aligned}
    \gamma & \le \inf_{q \in Q_H}\sup_{\vec v \in V_H} \frac{b(\vec v, q)}{\norm{q}_{Q} \norm{\vec v}_V} \\
    & \le (1+\alpha^{-1}\beta)\inf_{q \in Q_H}\sup_{\vec v \in V_H} \frac{(\nabla \cdot (\mbox{Id} - G_h)\vec v, q)}{\norm{q}_{Q} \norm{(\mbox{Id} - G_h)\vec v}_V} \\
    & = (1+\alpha^{-1}\beta)\inf_{q \in Q_H}\sup_{\vec v \in \VmsH} \frac{(\nabla \cdot \vec v, q)}{\norm{q}_{Q} \norm{\vec v}_V},
  \end{aligned}
\end{equation}
i.e.\ $b(\cdot,\cdot)$ is inf-sup stable with constant
$\gamma(1+\alpha^{-1}\beta)^{-1}$ independent of $H$ and $h$. We note
that $\KmsH = \{\vec v \in \VmsH : b(\vec v, q_H) = 0\
\forall q\in Q_H\}$, since $\nabla \cdot \vec v \in Q_H$ (see
Remark~\ref{rem:kernel}). Finally, we apply Lemma~\ref{lem:existence}
with $\mathcal{V} = \VmsH$, $\mathcal{Q} = Q_H$, $\mathcal{K} = \KmsH$
and constants $\tilde \alpha = \alpha$, $\tilde \beta = \beta$ and
$\tilde \gamma = \gamma(1+\alpha^{-1}\beta)^{-1}$.
\end{proof}

\subsection{Error estimate for ideal problem}
In this section, we show
that the flux solution of
the ideal multiscale problem above converges in the energy norm with linear order in $H$ 
to the reference solution. This convergence is independent of the variations of ${\vec A}$,
i.e.\ we do not have any pre-asymptotic effects from the multiscale features.
\begin{lemma}[Error estimate for ideal solution]
  \label{lem:idealerrorestimate}
  Under Assumptions~(A1)--(A3) and (B1)--(B3), let $\vec u_h$
  solve \eqref{eq:mixeddiscrete} and $\umsH$ solve
  \eqref{eq:msmixedbilinear}, then
\begin{equation*}
  \begin{aligned}
    \vertiii{\vec u_h - \umsH} \le{} & \beta^{1/2} C_{\widehat \Pi} C_{\rho,d} H \|f - P_H f\|_{\Ltwo}\\
  \end{aligned}
\end{equation*}
where $C_{\rho,d}$ and $C_{\widehat \Pi}$ are independent of $h$ and $H$.
\end{lemma}
\begin{proof}
  Parametrizing the solutions $\vec u_h(f)$ and
  $\umsH(f)$ by the data $f$, we use the triangle inequality to obtain
  \begin{multline*}
    \vertiii{\vec u_h(f) - \umsH(f)} \\
    \le
    \vertiii{\vec u_h(f) - \vec u_h(P_H f)} +
    \vertiii{\vec u_h(P_H f) - \umsH(P_H f)} +
    \vertiii{\umsH(P_H f) - \umsH(f)}.
  \end{multline*}
The two last terms will be shown to equal zero.

For the first term, we proceed in several steps. Let us define $\tilde {\vec u}_h := \vec u_h(f) - \vec u_h(P_Hf) = \vec u_h(f - P_Hf)$,
which is the flux solution for the data $f - P_Hf$. The corresponding pressure solution shall be denoted by $\tilde p_h$. First, we observe
\begin{equation}
  \label{lemma-8-proof-step-1}\begin{aligned}
    \vertiii{\tilde {\vec u}_h}^2 & = (f - P_Hf, \tilde p_h) = (f - P_Hf, \tilde p_h - P_H \tilde p_h) \le \|f - P_H f\|_{\Ltwo} \| \tilde p_h - P_H \tilde p_h \|_{\Ltwo}.
  \end{aligned}
\end{equation}
In order to bound the term $\| \tilde p_h - P_H \tilde p_h \|_{\Ltwo}$, we let
$\phi \in H^1_0(\Omega)$ be the weak solution
to $\Delta \phi = \tilde p_h - P_H \tilde p_h$. Then we have
\begin{equation*}
\begin{aligned}
    | \phi |_{H^1(\Omega)}^2 = (\tilde p_h - P_H \tilde p_h, \phi - P_H \phi) \le C_{\rho,d} H \|\tilde p_h - P_H \tilde p_h\|_{\Ltwo} |\phi|_{H^1(\Omega)}.
  \end{aligned}
\end{equation*}
Defining $\vec w := \nabla \phi$ we get $\nabla \cdot \vec w = \tilde
p_h - P_H \tilde p_h$ and $\| \vec w \|_{\Ltwo} \le C_{\rho,d} H
\|\tilde p_h - P_H \tilde p_h\|_{\Ltwo}$.  Next, we use a pair of
projection operators $\widehat {\Pi}_h : V \to V_h$ and $\widehat P_h
: Q \to Q_h$ that commute with respect to the divergence operator,
allows for $\trih$ to be non quasi-uniform, and where $\widehat
{\Pi}_h$ is $L^2$-stable, i.e.\ $\widehat P_h \nabla \cdot \vec w =
\nabla \cdot \widehat {\Pi}_h \vec w$ and $\| \widehat {\Pi}_h \vec w
\|_{\Ltwo} \le C_{\widehat \Pi} \| \vec w \|_{\Ltwo}$, with
$C_{\widehat \Pi}$ independent of $h$. The existence of such operators
is proved in \cite{ChWi08}. Exploiting this stability
and the fact that $\tilde p_h - P_H \tilde p_h = \widehat P_h (\nabla
\cdot \vec w)$ (since $\widehat P_h$ is a projection on $Q_h$ and
$\tilde p_h - P_H \tilde p_h\in Q_h$), we obtain
\begin{equation*}
  \begin{aligned}
    \| \tilde p_h - P_H \tilde p_h \|_{\Ltwo}^2 & = (\tilde p_h - P_H \tilde p_h, \tilde p_h) 
     = (\widehat P_h (\nabla \cdot \vec w), \tilde p_h) \\
    &= (\nabla \cdot \widehat {\Pi}_h \vec w, \tilde p_h)
    = -({\vec A}^{-1} \tilde {\vec u}_h, \widehat {\Pi}_h \vec w)
     \le \vertiii{\tilde {\vec u}_h} \|{\vec A}^{-1/2} \widehat {\Pi}_h \vec w\|_{\Ltwo} \\
     &\le \beta^{1/2} C_{\widehat \Pi} \vertiii{\tilde {\vec u}_h} \|\vec w\|_{\Ltwo} 
     \le \beta^{1/2} C_{\widehat \Pi} C_{\rho,d} H \vertiii{\tilde {\vec u}_h} \|\tilde p_h - P_H \tilde p_h\|_{\Ltwo}.
  \end{aligned}
\end{equation*}
Combining this estimate with \eqref{lemma-8-proof-step-1} yields
\begin{equation*}
    \vertiii{\tilde {\vec u}_h}^2 \le \beta^{1/2} C_{\widehat \Pi} C_{\rho,d} H \|f - P_H f\|_{\Ltwo} \vertiii{\tilde {\vec u}_h}.
\end{equation*}

For the second term, since the correctors are divergence free, we
have $\nabla \cdot \umsH(P_H f) \in Q_H$. This implies $\nabla \cdot \umsH(P_H f) = -P_H
f$, hence
\begin{equation*}
  \nabla \cdot \umsH(P_H f) - \nabla \cdot \vec u_h(P_H f) = 0,
\end{equation*}
i.e.\ $\umsH(P_H f) - \vec u_h(P_H f) \in K_h$. 
Now, from first the equations in
\eqref{eq:mixeddiscrete} and \eqref{eq:msmixedbilinear} in combination
with the $a(\cdot,\cdot)$-orthogonality between $\VmsH$ and $\Khf$, we get
\begin{equation*}
  \begin{aligned}
    a(\vec u_h(P_H f), \vec v) & = 0, &&\quad \vec v \in \Vh,&& \nabla \cdot \vec v = 0, \text{ and}\\
    a(\umsH(P_H f), \vec v) & = 0, &&\quad \vec v \in \VmsH,&& \nabla \cdot \vec v = 0, \text{ and}\\
    a(\umsH(P_H f), \vec v) & = 0, &&\quad \vec v \in \Vhf,&& \nabla \cdot \vec v = 0.\\
  \end{aligned}
\end{equation*}
Since $\Vh = \VmsH \oplus \Vhf$, we obtain
\begin{equation*}
  \begin{aligned}
    a(\vec u_h(P_H f) - \umsH(P_H f), \vec v) & = 0,\\
  \end{aligned}
\end{equation*}
for all $\vec v \in K_h$. Choosing $\vec v = \vec u_h(P_H f) - \umsH(P_H f)$, we see
that $\vec u_h(P_H f) = \umsH(P_H f)$, thus the second term equals zero.

To show that the third term is zero, it is sufficient to show that
$\umsH(f - P_H f) = 0$. The data $f - P_Hf$ is $L^2$-orthogonal to the
test space $Q_H$ and it enters the equation \eqref{eq:msmixedbilinear}
only in an $L^2$ scalar product with test functions. Hence $\umsH(f -
P_H f) = \umsH(P_H(f - P_H f)) = 0$.
\end{proof}

\section{Localized multiscale method}
\label{sec:localized}
The ideal corrector problems \eqref{eq:idealcorrector} are at least as
expensive to solve as the original reference problem.  Hence, we
require to localize these problems to very small patches, without
sacrificing the good approximation properties.  If we can achieve
this, the corrector problems can be solved with low computational
costs and fully in parallel.  In this section, we show that this is
indeed possible.  We prove that we can truncate the computational
domain $\Omega$ in the local corrector problems
\eqref{eq:idealcorrector} to a small environment of a coarse element
$T$.  This is possible, since the solutions of
\eqref{eq:idealcorrector} decay with exponential rate
outside the coarse element $T$. 
We obtain a new localized corrector operator
which can be used analogously to the ideal corrector operator to
construct a localized multiscale space. This localization reduces the
computational effort for assembling the multiscale space significantly. 
 
In addition to the assumptions (A1)--(A3) and (B1)--(B3), we require
additional assumptions on the computational domain and the mesh. More
precisely we assume the following for the analysis.
  \begin{enumerate}
  \item[{\upshape (A4)}] We consider $d=2$ and a simply-connected domain $\Omega \subset \R^2$.
  \item[{\upshape (B4)}] The fine grid $\mathcal{T}_h$ is quasi-uniform, i.e.\
    the ratio between the maximum and the minimum diameter of a grid
    element is bounded by a generic constant.
  \end{enumerate}
  We note that assumption (A4) is crucial for our proof. Assumption
  (B4) on the other hand could be dropped with a more careful
  analysis. In this case the estimates (and in particular the decay)
  will depend on the inverse of the minimum mesh size of the fine grid
  in a patch $U(T)$. For simplicity of the presentation, we do not
  elaborate this case and restrict ourselves to quasi-uniform meshes,
  i.e.\ to (B4). Note that even though we fix $d=2$, we keep the general notation $d$ to
  illustrate how the results are influenced by the dimension. The localized method can be formulated analogously for $d=3$.

In order to localize the detail space $\Khf$, we use admissible
patches. We call this restriction to patches {\it localization}.  For
each $T \in \T_H$ we pick a connected patch $U(T)$ consisting of
coarse grid elements and containing $T$. More precisely, for positive $k\in
\mathbb{N}$ we define $k$-coarse-layer patches iteratively in the
following way. For all $T\in\T_H$ (which are assumed to be closed sets), we define the element patch
$U_k(T)$ in the coarse mesh $\T_H$ by
\begin{equation}\label{def-patch-U-k}
    \begin{aligned}
      U_0(T) & := T, \\
      U_k(T) & := \bigcup\{T'\in \T_H : T'\cap U_{k-1}(T)\neq\emptyset\}\quad k=1,2,\ldots .
    \end{aligned}
\end{equation}
\begin{figure}[htb]
  \centering
  \begin{subfigure}{.4\textwidth}
    \centering
    \includegraphics[width=\textwidth]{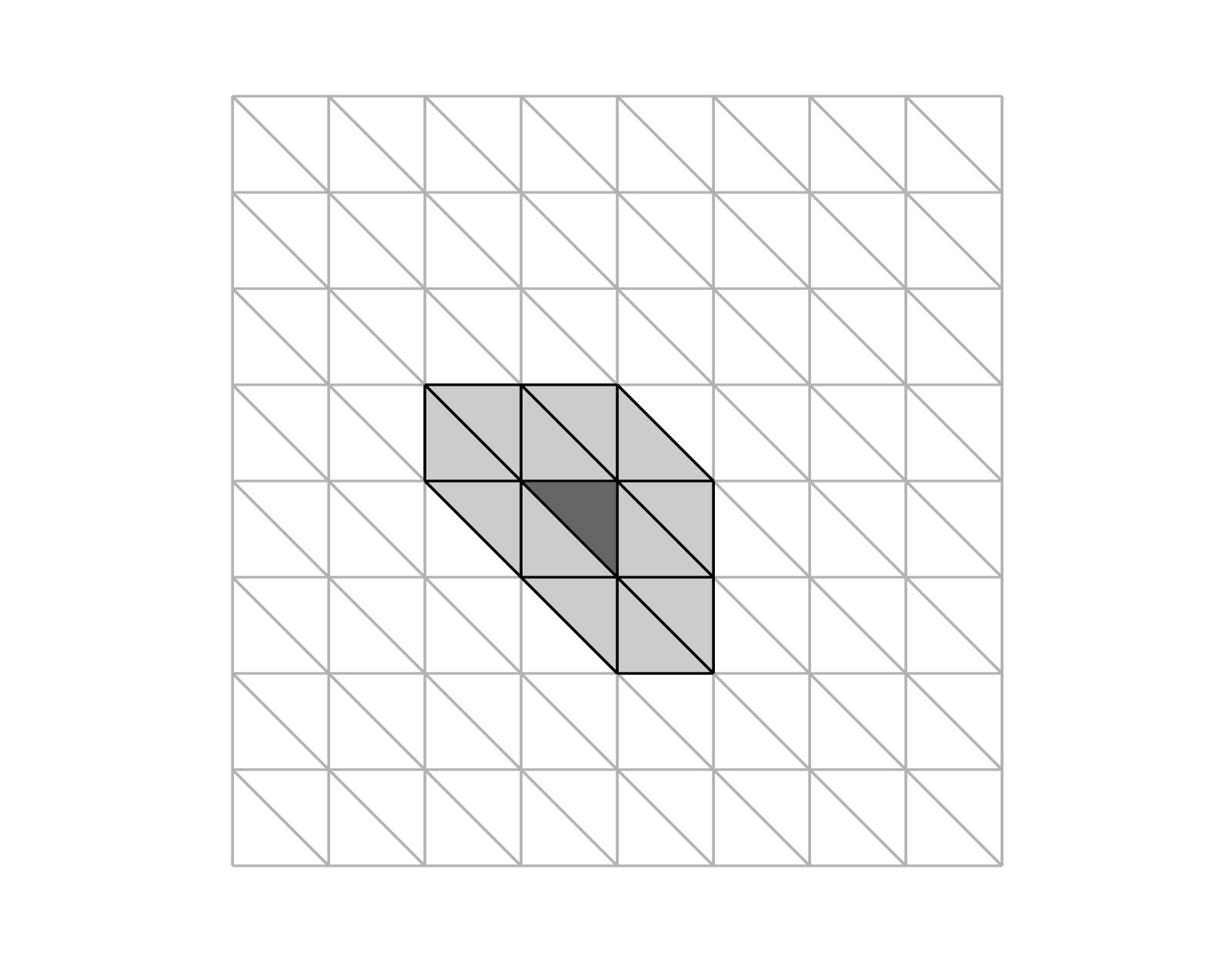}
    \caption{One-coarse-layer patch, $k = 1$.}
  \end{subfigure}
  \hspace{3em}
  \begin{subfigure}{.4\textwidth}
    \centering
    \includegraphics[width=\textwidth]{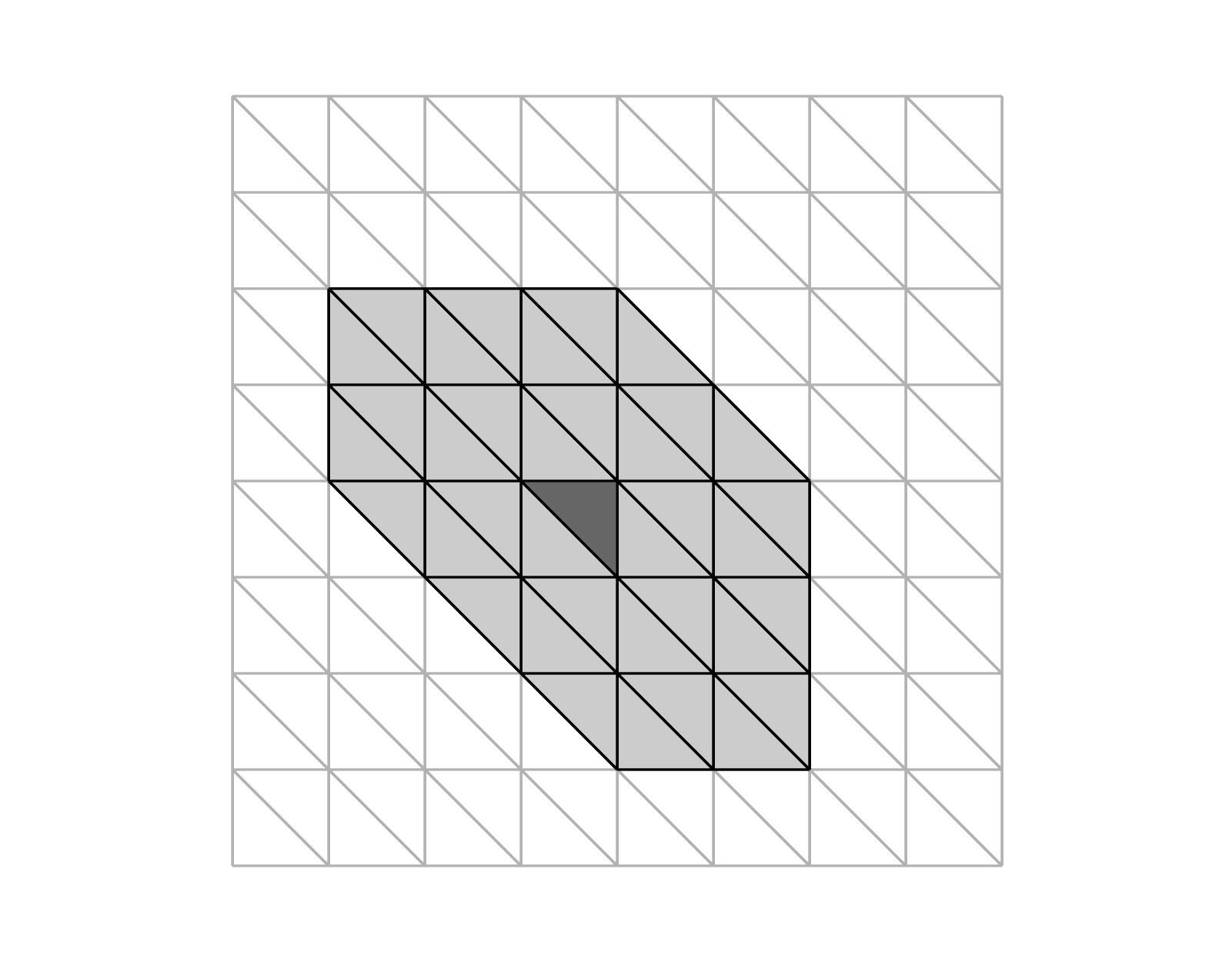}
    \caption{Two-coarse-layer patch, $k = 2$.}
  \end{subfigure}
  \caption{Illustration of $k$-coarse-layer patches. Dark gray subdomain is $T$.
    Light gray subdomain is $U_k(T)$.}
  \label{fig:patch}
\end{figure}
See Figure~\ref{fig:patch} for an illustration of patches. For a
given patch $U(T)$, we define the restriction of $\Vhf$ to $U(T)$ by
$$\oVhf{U(T)}:=\{ \vec w \in \Vhf : \vec w=0 \enspace \mbox{in } \Omega \setminus U(T) \}.$$ 
Accordingly, we also define
$$\oKhf{U(T)}:=\{\vec w \in \oVhf{U(T)} : \hspace{2pt}\nabla \cdot \vec w = 0 \}.$$ 
Using this localized space, we define the localized corrector
operators. Localized quantities are indexed by the patch layer
size $k$.
\begin{definition}[Localized corrector operators]
  \label{def:loccorrector}
  For each $T \in \triH$ and $k \ge 1$ layers, we define a \emph{localized
    element corrector operator} $G^T_{h,k} : V \to \oKhf{U_k(T)}$:
\begin{equation}
  \label{eq:loccorrector}
  a(G^T_{h,k}\vec v, \vec w) = a^T(\vec v, \vec w)
\end{equation}
for all $\vec w \in \oKhf{U_k(T)}$. Further, we define the
\emph{localized global corrector operator} $G_{h,k} := \sum_{T\in\triH}
G^T_{h,k}$.
\end{definition}
The localized corrector operators are again well-defined by 
the Lax-Milgram theorem, exploiting that $a(\cdot,\cdot)$ is a weighted $L^2$-scalar product. 
Note that the definition of
$\oKhf{U_k(T)}$ implies Neumann boundary conditions on the localized
corrector problems \eqref{eq:loccorrector}.  We define a localized
multiscale function space by
\begin{equation*}
  \label{eq:vmshk-definition}
  \VmsHk := (\mbox{Id}-G_{h,k})(V_H)
\end{equation*}
and state the localized multiscale problem as follows.
\begin{definition}[Localized multiscale problem]
  \label{def:locproblem}
The localized multiscale problem reads: find $\umsHk \in \VmsHk$ and
$\pH \in Q_H$, such that
\begin{equation}
  \label{eq:locmixedbilinear}
  \begin{aligned}
    a(\umsHk, \vec v_h) + b(\vec v_h, \pH) & = 0, \\
    b(\umsHk, q_H) & = -(f, q_H),
  \end{aligned}
\end{equation}
for all $\vec v_h \in \VmsHk$ and $q_H \in Q_H$.
\end{definition} 
Definitions~\ref{def:loccorrector} and \ref{def:locproblem} constitute
the proposed multiscale method. Next, we show that the above stated
problem is well-posed.
\begin{lemma}[Unique solution of localized multiscale problem]
  \label{lem:uniquelocalized}
  Under Assumptions~(A1)--(A3) and (B1)--(B3), the localized
  multiscale problem \eqref{eq:locmixedbilinear} has a unique solution
  for all $k$, $h$ and $H$.
\end{lemma}
\begin{proof}
We use similar arguments as in Lemma~\ref{lem:uniqueideal}.
The basic difference is that we need to show
  stability for the localized corrector operator $G_{h,k}$. 
  We start with the stability 
  of the
  localized element corrector operators. Here we have for arbitrary $\vec v \in V$
\begin{equation}
  \label{eq:GTkstab}
  \begin{aligned}
    \vertiii{G^T_{h,k} \vec v}^2 & = a(G^T_{h,k} \vec v, G^T_{h,k} \vec v)
     = a^T(\vec v, G^T_{h,k} \vec v)
     \le \vertiii{\vec v}_T \vertiii{G^T_{h,k} \vec v}.
  \end{aligned}
\end{equation}
Now, we can prove $L^2$-stability of the localized global operator. We get
\begin{equation*}
  \begin{aligned}
    \|G_{h,k} \vec v \|_{L^2(\Omega)}^2 & = \left\| \sum_{T\in\triH} G^T_{h,k} \vec v \right\|_{L^2(\Omega)}^2 
     \le \alpha^{-1} a\left (\sum_{T\in\triH} G^T_{h,k} \vec v, \sum_{T'\in\triH} G^{T'}_{h,k} \vec v \right)\\
    & = \alpha^{-1} \sum_{T\in\triH} \sum_{T' \subset U_k(T)} a(G^T_{h,k} \vec v,  G^{T'}_{h,k} \vec v )\\
    & \le \frac{1}{2} \alpha^{-1} \sum_{T\in\triH} \sum_{T' \subset U_k(T)} \left( \vertiii{G^T_{h,k} \vec v}^2 + \vertiii{G^{T'}_{h,k} \vec v}^2 \right) \\
    & \le \alpha^{-1} C_{\hspace{-1pt}\rho} k^{d}  \sum_{T\in\triH} \vertiii{G^T_{h,k} \vec v}^2
     \overset{\eqref{eq:GTkstab}}{\le} \alpha^{-1} C_{\hspace{-1pt}\rho} k^{d}  \sum_{T\in\triH} \vertiii{\vec v}_T^2
     \le \alpha^{-1} \beta C_{\hspace{-1pt}\rho} k^{d} \|\vec v\|_{L^2(\Omega)}^2,
  \end{aligned}
\end{equation*}
where $C_{\hspace{-1pt}\rho}$ is a constant only depending on the shape regularity
constant $\rho$ of the coarse mesh. Similar to \eqref{eq:infsup} we
derive inf-sup stability with
\begin{equation*}
  \begin{aligned}
    \gamma & \le \inf_{q \in Q_H}\sup_{\vec v \in V_H} \frac{b(\vec v, q)}{\norm{q}_{Q} \norm{\vec v}_V} \\
    & \le (1+\alpha^{-1/2} \beta^{1/2} C_{\hspace{-1pt}\rho}^{1/2} k^{d/2})\inf_{q \in Q_H}\sup_{\vec v \in \VmsHk} \frac{b(\vec v, q)}{\norm{q}_{Q} \norm{\vec v}_V}.
  \end{aligned}
\end{equation*}
Observe that the inf-sup stability constant $\gamma_k^0 :=
\gamma(1+\alpha^{-1/2} \beta^{1/2} C_{\hspace{-1pt}\rho}^{1/2} k^{d/2})^{-1}$
depends on $k$ this time.
\end{proof}

The inf-sup stability constant $\gamma_k^0$ depends on $k$ due to
overlapping patches. We come back to another estimate of the inf-sup
stability constant in Section~\ref{sec:infsup} after proving the decay
of the correctors.
 
It is important to note that in the localized case we do
not have orthogonality between $\VmsHk$ and $\Khf$ as in the ideal case
(cf.\ equation \eqref{eq:idealorthogonality}). This orthogonality was crucial in
the error estimate for the ideal method presented in
Lemma~\ref{lem:idealerrorestimate}. In the localized case, we rely on
the exponential decay of the localized element correctors, which
justifies localization to patches.

\subsection{Error estimate for localized problem}
In this section we state the main result of this paper, which is an a
priori error estimate in the energy norm between the reference
solution and the localized multiscale approximation. We first present
a logarithmic stability result for the nodal Raviart--Thomas
interpolation operator $\Pi_H$ for fine scale functions and then state
a lemma on the exponential decay of the correctors. Then the
main theorem follows. The proof of the exponential decay is contained in
Section~\ref{sec:decay}. The notation $a \lesssim b$ stands for $a \le
Cb$ with some constant $C$ that might depend on $d$, $\Omega$,
$\alpha$, $\beta$ and coarse and fine mesh regularity constants, but
not on the mesh sizes $h$ and $H$. In particular it does not depend on
the possibly rapid oscillations in ${\vec A}$.

We recall a well known stability result for the nodal Raviart--Thomas interpolation operator.
\begin{lemma}[Logarithmic stability of the nodal interpolation
  operator for divergence free functions]
\label{lemma-wohlmuth}
Assume (B1)--(B4). For any given element $T \in \T_H$ there exists a
constant $C$ that only depends on the regularity of $T$ and the
quasi-uniformity of $\trih$, such that
$$\| \IH \vec v_h \|_{L^2(T)}^2 \le C \lambda(H/h)^2 \| \vec v_h \|_{L^2(T)}^2,$$
with $\lambda(H/h) := (1+\log(H/h))^{1/2}$ for all $\vec v_h \in \Vh$ with $\nabla \cdot \vec v_h = 0$.
\end{lemma}
A proof for this can be found in \cite[Lemma 4.1]{WoToWi00}. This
result holds for both $d=2$ and $3$. 
\begin{remark}
  \label{rem:clement}
  There exist unconditionally $L^2$-stable Cl\'ement-type
  interpolation operators for which we could define $\lambda(H/h) :=
  1$ for all $h$ and $H$ instead, see \cite{DoFaWi06, Chr05, ChWi08,
    Sch05}. In particular, the operators introduced in \cite{DoFaWi06,
    ChWi08} are projections and were used as a technical tool in the
  proof of Lemma~\ref{lem:idealerrorestimate} above. However, these
  operators are hard to implement in practice and hence are not used
  in the proposed numerical method.
\end{remark}

\begin{lemma}[Exponential decay of correctors]
  \label{trunc-applied}
  Under Assumptions~(A1)--(A4) and (B1)--(B4), there exists a generic
  constant $0 < \theta < 1$ depending on the contrast
  $\beta/\alpha$, but not on $h$ or $H$ such that for all positive
  $k\in\mathbb{N}$:
  \begin{align}
    \vertiii{\sum_{T \in \T_H} \left(G_h^T\vec v-G_{h,k}^T\vec v\right)}^2 \lesssim k^d \lambda(H/h)^2 \theta^{2k/\lambda(H/h)} \sum_{T \in \T_H} \vertiii{G_h^T \vec v}^2
  \end{align}
  for all $\vec v \in V$.
\end{lemma}
\begin{proof} The lemma is a direct consequence of Lemma~\ref{trunc-lemma-3} in Section~\ref{sec:decay}.
\end{proof}

Now, combining the error estimate for the ideal multiscale method in
Lemma~\ref{lem:idealerrorestimate} and Lemma~\ref{trunc-applied} we
get the following a priori error estimate of the localized multiscale method.
\begin{theorem}[Error estimate for localized multiscale solution]
  \label{thm:errorestimate}
  Under Assumptions (A1)--(A4) and (B1)--(B4), for a positive $k \in
  \mathbb{N}$, let $\vec u_h$ solve \eqref{eq:mixeddiscrete} and
  $\umsHk$ solve \eqref{eq:locmixedbilinear}, then
  \begin{equation}\label{final-estimate}
    \vertiii{\vec u_h - \umsHk} \lesssim H\|f - P_H f\|_{L^2(\Omega)} + k^{d/2} \lambda(H/h)^2 \theta^{k/\lambda(H/h)} \|f\|_{L^2(\Omega)},
  \end{equation}
  for some $0 < \theta < 1$ depending on the contrast $\beta/\alpha$, but
  not on $k$, $h$ and $H$.
\end{theorem}
Before stating the proof, we discuss the role and choice of $k$. The
second term in the error estimate \eqref{final-estimate} is an effect
of the localization. This term can be made small by choosing large
values of $k$, i.e.\ large patch sizes. A natural question is how to
choose $k$ to make the second term of order $H$ to some power.

We write $\lambda = \lambda(H/h)$ for convenience. Let $\tilde k =
2d^{-1}\log(\theta) \lambda^{-1}k = -C_\theta\lambda^{-1}k$, where
$C_\theta = -2d^{-1}\log(\theta) > 0$ is a constant independent of $H$
and $h$. We are interested in the asymptotic behavior, so we consider
$H \ll 1$. Setting the second term in \eqref{final-estimate} equal to
$H\|f\|_{L^2(\Omega)}$ yields
\begin{equation*}
   \tilde k e^{\tilde k} = -C_\theta \lambda^{-4/d-1} H^{2/d},
\end{equation*}
that is $\tilde k = W(-C_\theta \lambda^{-4/d-1} H^{2/d})$, where $W$
is the Lambert $W$-function. In terms of the number of layers $k$, we
get $k = -C_\theta^{-1}\lambda W(-C_\theta \lambda^{-4/d-1}
H^{2/d})$. This equation has two solutions for sufficiently small
$H$. Since we require $k \ge 1$, we pick the branch $W \le -1$.

Another, more practical option is to choose $k = R\lambda\log(1/H)$
for some constant ${R}$. Then the expression
$k^{d/2}\lambda\theta^{k/\lambda}$ will be asymptotically (as $H \to
0$) dominated by the power $H^{-{R}\log\theta}$. Choosing ${R}$
sufficiently large yields arbitrary order of accuracy of the term.  The
fine mesh size $h$ is often fixed and we can choose
\begin{equation}
  \label{eq:choosek}
  k = (1+|\log_r(H)|)^{1/2}\log_s(1/H)
\end{equation}
for some bases $r$ and $s$ of the two logarithms.

\begin{remark}
  If Cl\'ement-type interpolation operators are used, we have
  $\lambda \equiv 1$ independent of $H/h$. Choosing $k =
  C\log(1/H)$ makes the second term in \eqref{final-estimate}
  proportional to $\log(1/H)^{d/2} H^{-C \log \theta}$. For an
  appropriate $C$ we can make the first term in \eqref{final-estimate}
  dominate the error estimate.
\end{remark}

\begin{proof}[Proof of Theorem~\ref{thm:errorestimate}.]
  Let $\tildeumsH := ((\mbox{Id} - G_{h,k})\circ \Pi_H) \umsH \in \VmsHk$, then
  $\tildeumsH - \umsHk$ is divergence free. Hence, by Galerkin
  orthogonality we have
  \begin{equation*}
    a(\vec u_h - \umsHk, \vec u_h - \umsHk) = a(\vec u_h - \umsHk, \vec u_h - \tildeumsH)
  \end{equation*}
  and obtain
  \begin{align*}
    \vertiii{\vec u_h - \umsHk} \le \vertiii{\vec u_h - \tildeumsH} \le \vertiii{\vec u_h - \umsH} + \vertiii{\umsH - \tildeumsH}.
  \end{align*}
  The first term can be bounded by $\beta^{1/2} C_{\widehat \Pi}
  C_{\rho,d} H \| f - P_Hf\|_{\Ltwo}$ by
  Lemma~\ref{lem:idealerrorestimate}. Regarding the second term, using
  \cite[Lemma 4.1]{WoToWi00} and stability of the ideal multiscale
  solution, we get
  \begin{align*}
     \sum_{T \in \T_H} \vertiii{G^T_{h} \Pi_H \umsH}^2 \le \sum_{T \in \T_H} \vertiii{\Pi_H \umsH}_T^2 = \vertiii{\Pi_H \umsH}^2 \lesssim \lambda(H/h)^2\|f\|^2_{L^2(\Omega)}
   \end{align*}
   and can combine this with Lemma~\ref{trunc-applied} to get
  \begin{align*}
    \vertiii{\umsH - \tildeumsH} &= \vertiii{(G_{h,k}-G_h)\Pi_H \umsH} \\
    & = \vertiii{\sum_{T\in\triH} (G^T_{h,k}-G^T_h)\Pi_H \umsH} \\
    & \lesssim k^{d/2} \lambda(H/h) \theta^{k/\lambda(H/h)} \left(\sum_{T \in \T_H} \vertiii{G^T_{h} \Pi_H \umsH}^2\right)^{1/2} \\
    & \lesssim k^{d/2} \lambda(H/h)^2 \theta^{k/\lambda(H/h)} \|f\|_{L^2(\Omega)}.
  \end{align*}
\end{proof}

\subsection{Proof of exponential decay of correctors}
\label{sec:decay}
This section consists of four lemmas,
Lemma~\ref{existence-of-skew-symmetric-matrix}--\ref{trunc-lemma-3},
of which the last one is the main result. The two first lemmas are
auxiliary and are motivated by steps in the proofs of the latter two.
Before starting, we need to set some notation and introduce some
tools. We use the notation $W^{1,2}_{\mathrm{loc}}(\R^d) = \{ f : f
\in H^1(\omega) \ \forall \text{ compact subsets } \omega \subset
\R^d\}$. Note that we will use the letter $K$ to denote arbitrary
triangles of the coarse mesh $\triH$. The first lemma says that every
divergence free function $\vec w$ in $\Hdiv$ is the divergence of a
skew-symmetric matrix.

\begin{lemma}\label{existence-of-skew-symmetric-matrix}
  Let $\Omega$ be a simply connected domain with Lipschitz boundary
  and let $\vec w \in \Hdiv$ with $\nabla \cdot \vec w = 0$ in
  $\Omega$. Then there exists a skew-symmetric matrix $\psi \in
  [W^{1,2}_{\mathrm{loc}}(\R^d)]^{d \times d}$ with $\nabla \psi_{ij}
  \in [L^2(\R^d)]^d$ and $\int_{\Omega} \psi = 0$ such that
\begin{align}
  \label{properties-of-skew-symmetric-matrix}\vec w = \nabla \cdot
  \psi \quad \mbox{in } \Omega \qquad \mbox{and} \qquad \| \nabla \psi_{ij}
  \|_{L^2(\omega)} 
  \lesssim 
  \| \vec w \|_{L^2(\omega)} \qquad
  \mbox{for } \omega \subset \Omega.
\end{align}
Here, the divergence of $\psi$ is defined along the rows.
\end{lemma}
Note that the above lemma is the only instance, where we require the restriction $d=2$. Even though the existence of the skew-symmetric matrix is also available for $d=3$, we could not prove localized estimates of the type $\| \nabla \psi_{ij} \|_{L^2(\omega)} \lesssim \| \vec w \|_{L^2(\omega)}$.
\begin{proof}
  The result is a combination of well-known results. First, we extend
  the divergence-free vector field $\vec w \in \Hdiv$ to a
  divergence-free vector field $\tilde{\vec w} \in \HdivRd$. In
  particular we have $\tilde{\vec w}\in [L^2(\R^d)]^d$ and $\tilde{\vec
    w}=\vec w$ in $\Omega$. Note that the extension of $\vec w$ to
  $\mathbb{R}^d$ will be typically not zero outside of $\Omega$. The
  existence of such an extension operator was proved in
  \cite[Proposition 3.8]{Wie09}. It is well known that there exists a
  skew-symmetric matrix $\psi \in [W^{1,2}_{\mathrm{loc}}(\R^d)]^{d
    \times d}$ with $\nabla \psi_{ij} \in [L^2(\R^d)]^d$, such that
  $\tilde{\vec w} = \nabla \cdot \psi$ (see \cite[Lemma
  2.3]{IKL13}). The matrix is only unique up to a constant, so we
  fix the constant by $\int_{\Omega} \psi = 0$ (which gives us a
  Poincar\'e inequality). The 
  inequality $\| \nabla \psi_{ij} \|_{L^2(\omega)} \lesssim \| \tilde{\vec w} \|_{L^2(\omega)}$
  (for $\omega \subset \R^d$) 
   can be seen as follows for $d=2$.
  Obviously, if $i=j$ we obtain $\nabla \psi_{ii}=\nabla \psi_{jj}=0$ and estimate \eqref{properties-of-skew-symmetric-matrix} is trivial. If $i\neq j$, we obtain by using the skew-symmetry
  \begin{align*}
  \| \vec w \|_{L^2(\omega)}^2 &= \| \nabla \cdot
  \psi \|_{L^2(\omega)}^2 = \| \partial_1 \psi_{11} + \partial_2 \psi_{12} \|_{L^2(\omega)}^2 
  + \| \partial_1 \psi_{21} + \partial_2 \psi_{22} \|_{L^2(\omega)}^2\\
  &=
  \| \partial_2 \psi_{12} \|_{L^2(\omega)}^2 
  + \| \partial_1 \psi_{21}  \|_{L^2(\omega)}^2
  =
  \| \partial_2 \psi_{12} \|_{L^2(\omega)}^2 
  + \| \partial_1 \psi_{12}  \|_{L^2(\omega)}^2\\
  &= \| \nabla  \psi_{12}  \|_{L^2(\omega)}^2 = \| \nabla  \psi_{21}  \|_{L^2(\omega)}^2,
  \end{align*}
  i.e. we obtain even equality in estimate \eqref{properties-of-skew-symmetric-matrix}.
\end{proof}

We also require suitable cut-off functions that are central for the
proof. For $T\in\T_H$ and positive $k\in\mathbb{N}$, we let the
function $\eta_{T,k} \in P_1(\triH)$ (globally continuous and piecewise linear w.r.t.\
$\triH$) be defined as
\begin{equation}\label{e:cutoffH}
\begin{aligned}
 \eta_{T,k}(x) &= 0\quad\text{for } x \in U_{k-1}(T),\\
 \eta_{T,k}(x) &= 1\quad\text{for } x \in  \Omega\setminus U_k(T).\\
\end{aligned}
\end{equation}
  
We start with the following lemma, which enables us to approximate truncated functions from $\Khf$.

\begin{lemma}
\label{trunc-lemma-1}
Let $\vec w_h \in \Khf$ and let $\psi \in
[W^{1,2}_{\mathrm{loc}}(\Omega)]^{d \times d}$ with $\vec w_h = \nabla
\cdot \psi$ denote the corresponding skew-symmetric matrix as in Lemma
\ref{existence-of-skew-symmetric-matrix}. Let furthermore $\psi_K :=
|K|^{-1}\int_{K} \psi$ denote the average on $K \in \T_H$ and let
$\psi_H \in [L^2(\Omega)]^{d \times d}$ denote the corresponding
piecewise constant matrix with $\psi_H(x)=\psi_K$ for $x \in K$. The
broken divergence-operator $\nabla_H \cdot {}$ 
is given by $\nabla_H \cdot v := \nabla \cdot v \vert_{K}$
for $K \in \T_H$. The function $\eta_{T,k}\in P_1(\T_H)$ is a
given cut-off function as defined in \eqref{e:cutoffH} for
$k>0$. Then, we have that the function $\tilde{\vec w}_h:= \Ih \left(
  \nabla \cdot (\eta_{T,k} \psi) \right) - (\IH \circ \Ih) \left(
  \nabla \cdot (\eta_{T,k} \psi) \right) \in \Khf$ fulfills the
following estimate for any $K \in \triH$:
\begin{multline*}
      \| \nabla \cdot (\eta_{T,k} \psi )
      - \nabla_H \cdot (\eta_{T,k} \psi_H ) 
      - \tilde{\vec w}_h \|_{L^2(K)}
    \\
    \lesssim
    \begin{cases}
      \lambda(H/h) \| \vec w_h \|_{L^2(K)} & K \subset U_{k}(T) \setminus U_{k-1}(T)\\
      0 & \text{otherwise}.
    \end{cases}
\end{multline*}
Obviously we also have $\mbox{\rm supp}(\tilde{\vec w}_h) \subset \Omega \setminus U_{k-1}(T)$.
\end{lemma}

\begin{proof}
First, we observe that the skew-symmetric matrix $\psi$ 
must be a polynomial of maximum degree $2$ on each fine grid element. We use this in the following without mentioning. 

We fix the element $T\in\T_H$ and $k\in\mathbb{N}$ and denote $\eta:=\eta_{T,k}$. Furthermore, we define for $K \in \T_H$
\begin{align*}
c_K := |K|^{-1}\int_{K} \eta \qquad \mbox{and} \qquad \psi_K :=  |K|^{-1}\int_{K} \psi.
\end{align*}
We define $\tilde{\vec w}_h:= \Ih \left( \nabla \cdot (\eta \psi) \right) - (\IH \circ \Ih) \left( \nabla \cdot (\eta \psi) \right)$ and observe that $\tilde{\vec w}_h \in \Khf$ and $\vec w_h = \tilde{\vec w}_h$ on $\Omega \setminus U_k(T)$.
The property $\IH(\tilde{\vec w}_h)=0$ is clear. The property $\nabla \cdot \tilde{\vec w}_h =0$ follows from the fact that $\eta \psi$ is still skew symmetric and that $\nabla \cdot (\IH \circ \Ih)(\cdot) = (P_H \circ P_h)( \nabla \cdot{})$. Since $\psi_K$ and $c_K$ are constant on $K$ we have 
\begin{align}
\label{trunc-lemma-1-proof-step-1}\Ih \left( \nabla \cdot (c_K \psi_K) \right) = \nabla \cdot (c_K \psi_K) = 0 \quad \mbox{on } K.
\end{align}
Furthermore, since $\IH(\vec v_H)=\vec v_H$ for all $\vec v_H \in V_H$ and since $\nabla \cdot ( \eta \psi_K ) \in V_H$ we also have
\begin{align}
\label{trunc-lemma-1-proof-step-2}(\IH \circ \Ih)( \nabla \cdot ( \eta \psi_K ) ) =  \Ih \left( \nabla \cdot ( \eta \psi_K ) \right) \quad \mbox{on } K.
\end{align}
Finally, we also have on $K$,
\begin{align}
\label{trunc-lemma-1-proof-step-3}(\IH \circ \Ih)(  \nabla \cdot (c_K \psi ) ) = c_K  (\IH \circ \Ih)(  \nabla \cdot \psi ) = c_K  \IH(  \vec w_h ) = 0.
\end{align}
Combining \eqref{trunc-lemma-1-proof-step-1}, \eqref{trunc-lemma-1-proof-step-2}, and \eqref{trunc-lemma-1-proof-step-3} we obtain for every $K \in \T_H$
\begin{align}
\label{trunc-lemma-1-proof-step-4}
\nonumber &\| (\IH \circ \Ih) \left( \nabla \cdot (\eta \psi ) \right) - \Ih ( \nabla \cdot ( \eta \psi_K ) ) \|_{L^2(K)}\\ 
\nonumber &\quad = \| (\IH \circ \Ih) \left( \nabla \cdot (\eta \psi ) - \nabla \cdot (c_K \psi ) - \nabla \cdot ( \eta \psi_K ) + \nabla \cdot (c_K \psi_K) \right) \|_{L^2(K)} \\
&\quad = \|  (\IH \circ \Ih) \left( \nabla \cdot ((\eta-c_K) (\psi-\psi_K)) \right)  \|_{L^2(K)}.
\end{align}
Now, we consider the quantity we want to estimate. For any $K \in \T_H$,
\begin{align}
\label{trunc-lemma-1-proof-step-5}
\nonumber 
&\| \nabla \cdot (\eta \psi )
- \nabla_H \cdot (\eta \psi_H ) 
 - \tilde{\vec w}_h \|_{L^2(K)}\\[-1ex]
\nonumber &\quad \overset{\phantom{(99)}}{\le}
\|  \nabla \cdot (\eta (\psi - \psi_K ))
- \Ih \left( \nabla \cdot (\eta (\psi - \psi_K) ) \right) \|_{L^2(K)} \\
\nonumber&\quad \phantom{=} \quad {}+ \| \Ih \left( \nabla \cdot (\eta (\psi - \psi_K) ) \right)
 - \Ih \left( \nabla \cdot (\eta \psi) \right) + (\IH \circ \Ih) \left( \nabla \cdot (\eta \psi) \right)
   \|_{L^2(K)}\\
\nonumber &\quad \overset{\phantom{(99)}}{=} \|\nabla \cdot (\eta (\psi - \psi_K ))
- \Ih \left( \nabla \cdot (\eta (\psi - \psi_K) ) \right) \|_{L^2(K)}\\
\nonumber&\quad \phantom{=} \quad {}+ \| (\IH \circ \Ih) \left( \nabla \cdot (\eta \psi ) \right) - \Ih ( \nabla \cdot ( \eta \psi_K ) ) \|_{L^2(K)}\\
\nonumber &\quad \overset{\eqref{trunc-lemma-1-proof-step-4}}{=}
\|  \nabla \cdot ((\eta - c_K) (\psi - \psi_K ))
- \Ih \left( \nabla \cdot ((\eta-c_K) (\psi - \psi_K) ) \right) \|_{L^2(K)}\\
\nonumber&\quad \phantom{=} \quad {}+ \|  (\IH \circ \Ih) \left( \nabla \cdot ((\eta-c_K) (\psi-\psi_K)) \right)  \|_{L^2(K)}\\[-1ex]
&\quad \overset{\phantom{(99)}}{\lesssim} \lambda(H/h) \|  \nabla \cdot ((\eta - c_K) (\psi - \psi_K )) \|_{L^2(K)}.
\end{align}
In the last step we used Lemma \ref{lemma-wohlmuth}, the property that
$ \Ih \nabla \cdot ((\eta - c_K) (\psi - \psi_K ))$ is divergence free
and the fact that $\Ih$ is locally $L^2$-stable when applied to
functions of small fixed polynomial degree, i.e.\ for fixed $t\in\T_h$
and $r \in \mathbb{N}$ there exists a constant $C(r)$ that only
depends on $r$ and the shape regularity of $t$ such that
\begin{align*}
\| \Ih ( \vec v ) \|_{L^2(t)} \le C(r) \| \vec v \|_{L^2(t)} \qquad \mbox{for all } \vec v \in [\mathbb{P}^r(t)]^d.
\end{align*}
Continuing from \eqref{trunc-lemma-1-proof-step-5} we obtain
\begin{align}
\label{eq:limited-support}
\nonumber & \|  \nabla \cdot ((\eta - c_K) (\psi - \psi_K ) \|_{L^2(K)}^2 \\[-2ex]
\nonumber & \quad \overset{\phantom{(99)}}{\lesssim} \| (\eta-c_K) \nabla \cdot \psi \|_{L^2(K)}^2 + \| (\psi - \psi_K) \nabla \eta \|_{L^2(K)}^2 \\[-2ex]
\nonumber & \quad \overset{\phantom{(99)}}{\lesssim} H^2 \| \nabla \eta \|_{L^{\infty}(K)}^2  \| \nabla \psi  \|_{L^2(K)}^2 \\
& \quad \overset{\eqref{properties-of-skew-symmetric-matrix}}{\lesssim} 
\begin{cases}
 \| \vec w \|_{L^2(K)}^2 & K \subset U_{k}(T) \setminus U_{k-1}(T)\\
      0 & \text{otherwise}.
\end{cases}
\end{align}
Note that we used the properties of $\eta$ to obtain the 
Lipschitz bound $\| \eta- c_K \|_{L^{\infty}(K)} \lesssim H \| \nabla
\eta \|_{L^{\infty}(K)} \lesssim 1$ and that $\nabla \eta$ has no
support outside $U_{k}(T) \setminus U_{k-1}(T)$. We also used the
Poincar\'e inequality 
for $\eta- c_K$ which has a zero average on $K$.
Combining
\eqref{trunc-lemma-1-proof-step-5} and \eqref{eq:limited-support}
yields the sought result.
\end{proof}

We continue with a lemma showing the exponential decay of solutions to
problems of the form in \eqref{eq:idealcorrector}.

\begin{lemma}\label{trunc-lemma-2}
Now, let $\vec w^T \in \Khf$ be the solution of
\begin{align}
\label{generalized-corrector-problem-2-L2}\int_{\Omega} {\vec A}^{-1}  \vec w^T \cdot \vec v_h =F_T(\vec v_h) \qquad \mbox{for all } \vec v_h \in \Khf
\end{align}
where $F_T\in (\Khf)^{\prime}$ is such that $F_T(\vec v_h)=0$ for all
$\vec v_h \in \oKhf{\Omega \setminus T}$. Then, there exists a generic
constant $0<\theta < 1$ (depending on the contrast $\beta/\alpha$)
such that for all positive $k\in\mathbb{N}$:
\begin{align}
\label{lemma-a-2-eq-L2}\vertiii{\vec w^T}_{\Omega\setminus U_k(T)}\lesssim \theta^{k/\lambda(H/h)}\vertiii{\vec w^T}_{\Omega}.
\end{align}
\end{lemma}

\begin{proof}
The proof exploits similar arguments as in \cite{Pet14}.
Let us fix $k \in \mathbb{N}$. We denote again $\eta:=\eta_{T,k}\in P_1(\T_H)$ (as in \eqref{e:cutoffH}). We apply Lemma \ref{trunc-lemma-1} to $\vec w^T \in \Khf$. The corresponding skew symmetric matrix shall again be denoted by $\psi=\psi(\vec w^T)$ and we define
$$\tilde{\vec w}^T:=\Ih (\nabla \cdot (\eta \psi)) - (\IH \circ \Ih) \left( \nabla \cdot (\eta \psi) \right).$$
We obtain that $\nabla \cdot (\eta \psi ) - \nabla_H \cdot (\eta \psi_H ) -
\tilde{\vec w}^T$ is zero outside $U_{k}(T)\setminus U_{k-1}(T)$ and
\begin{multline}
\label{trunc-lemma-2-proof-step-0}\| \nabla \cdot (\eta \psi )
- \nabla_H \cdot (\eta \psi_H ) 
 - \tilde{\vec w}^T\|_{L^2(U_{k}(T)\setminus U_{k-1}(T))}  \lesssim \lambda(H/h) \| \vec w^T \|_{L^2(U_{k}(T)\setminus U_{k-1}(T))}.
\end{multline}
First observe that
\begin{align}
\label{trunc-lemma-2-proof-step-1} \int_{\Omega \setminus U_{k-1}(T)} {\vec A}^{-1} \vec w^T \cdot \tilde{\vec w}^T =  \int_{\Omega} {\vec A}^{-1} \vec w^T \cdot \tilde{\vec w}^T = F_T ( \tilde{\vec w}^T) = 0
\end{align}
and
\begin{align}
\label{trunc-lemma-2-proof-step-2}\eta \vec w^T = \eta \nabla \cdot \psi = \nabla \cdot ( \eta \psi ) - \psi \nabla \eta.
\end{align}
With that we have 
\begin{equation*}
\begin{aligned}
\int_{\Omega\setminus U_k(T)}{\vec A}^{-1} \vec w^T \cdot \vec w^T &\overset{\phantom{(99)}}{\leq} \int_{\Omega\setminus U_{k-1}(T)} {\vec A}^{-1}\vec w^T \cdot (\eta \vec w^T)\\
&\overset{\eqref{trunc-lemma-2-proof-step-2}}{=} \int_{\Omega\setminus U_{k-1}(T)} {\vec A}^{-1}\vec w^T \cdot \left( \nabla \cdot ( \eta \psi ) - \psi \nabla \eta \right) \\
&\overset{\eqref{trunc-lemma-2-proof-step-1}}{=} \int_{\Omega\setminus U_{k-1}(T)} {\vec A}^{-1}\vec w^T \cdot \left( \nabla \cdot ( \eta \psi ) - \psi \nabla \eta - \tilde{\vec w}^T \right)\\
&\overset{\phantom{(99)}}{=} \underset{=:\mbox{\rm I}}{\underbrace{\int_{\Omega\setminus U_{k-1}(T)} {\vec A}^{-1}\vec w^T \cdot \left( \nabla \cdot ( \eta \psi ) - \nabla_H \cdot (\eta \psi_H ) - \tilde{\vec w}^T\right)}} \\
&\phantom{=}\quad {}+ \underset{=:\mbox{\rm II}}{\underbrace{\int_{\Omega\setminus U_{k-1}(T)} {\vec A}^{-1}\vec w^T \cdot \left(\nabla_H \cdot (\eta \psi_H ) - \psi \nabla \eta \right)}}.\\
\end{aligned}
\end{equation*}
For I we use \eqref{trunc-lemma-2-proof-step-0} to obtain
\begin{align*}
\mbox{\rm I} 
\lesssim \lambda(H/h) \vertiii{\vec w^T}_{U_{k}(T) \setminus U_{k-1}(T)}^2
\end{align*}
and for II we obtain
\begin{equation*}
  \begin{aligned}
    \mbox{\rm II} &= \int_{\Omega\setminus U_{k-1}(T)} {\vec A}^{-1}\vec w^T \cdot \left( (\psi_H  - \psi) \nabla \eta \right)\\
    &\lesssim \underset{K \subset U_{k}(T) \setminus U_{k-1}(T)}{\sum_{K \in \T_H}} \vertiii{\vec w^T}_K H \| \nabla \eta \|_{L^{\infty}(K)} \| \nabla \psi \|_{L^2(K)} \\
    &\lesssim \vertiii{\vec w^T}_{U_{k}(T) \setminus U_{k-1}(T)}^2.
  \end{aligned}
\end{equation*}
Now, denote by $L := C\lambda(H/h)$, and we get
\begin{align*}
\vertiii{\vec w^T}_{\Omega\setminus U_{k}(T)}^2 & \leq L \vertiii{\vec w^T}_{U_{k}(T) \setminus U_{k-1}(T)}^2  \leq L \left(\vertiii{\vec w^T}_{\Omega \setminus U_{k-1}(T)}^2 - \vertiii{\vec w^T}_{\Omega \setminus U_{k}(T)}^2\right)
\end{align*}
where $C$ is independent of $T$, $k$ and ${\vec A}$, but can depend on
the contrast. 
We obtain
\begin{equation*}
  \vertiii{\vec w^T}_{\Omega\setminus U_{k}(T)}^2 \le (1+L^{-1})^{-1} \vertiii{\vec w^T}_{\Omega\setminus U_{k-1}(T)}^2.
\end{equation*}
A recursive application of this inequality and $\vertiii{\vec w^T}_{\Omega\setminus U_{0}(T)} \le \vertiii{\vec w^T}_{\Omega}$ yields
\begin{align*}
\vertiii{\vec w^T}_{\Omega\setminus U_{k}(T)}^2 & \le e^{-\log(1+L^{-1})k} \vertiii{\vec w^T}_{\Omega}^2  \le e^{-\log(1+C^{-1})k/\lambda(H/h)} \vertiii{\vec w^T}_{\Omega}^2,
\end{align*}
where we used Bernoulli's inequality and that $0 < L^{-1} \le C^{-1}$
in the last step. The choice $\theta:=(1+C^{-1})^{-1}$ proves the
lemma.
\end{proof}

The following lemma is the main result of this subsection. It can be
directly applied to the localized corrector problems
\eqref{eq:loccorrector} with $F_T(\vec v_h) = a^T(\vec v, \vec v_h)$,
$G^T_{h,k} \vec v = \vec w^{T,k}$ and $G^T_{h} \vec v = \vec w^T$ for any $\vec v \in V$.
\begin{lemma}
  \label{trunc-lemma-3}
Let the setting of Lemma \ref{trunc-lemma-2} hold true and let additionally $\vec w^{T,k} \in \oKhf{U_k(T)}$ denote the solution of
\begin{align}
\label{generalized-corrector-problem-localized-L2}\int_{U_k(T)} {\vec A}^{-1} \vec w^{T,k} \cdot \vec v_h =F_T(\vec v_h) \qquad \mbox{for all } \vec v_h \in \oKhf{U_k(T)}.
\end{align}
Then, there exists a generic constant $0 < \theta < 1$ (depending on the contrast) such that for all positive $k\in\mathbb{N}$:
\begin{align}
\vertiii{\sum_{T \in \T_H} \left(\vec w^T-\vec w^{T,k}\right)}_{\Omega}^2 \lesssim k^d \lambda(H/h)^2 \theta^{2k/\lambda(H/h)} \sum_{T \in \T_H} \vertiii{\vec w^T}_{\Omega}^2.
\end{align}
\end{lemma}
\begin{proof}
Let $\eta_{T,k}$ be defined according to \eqref{e:cutoffH} and denote $\vec z:=\sum_{T \in \T_H}(\vec w^T-\vec w^{T,k}) \in \Khf$.  We obtain
\begin{equation*}
\vertiii{\vec z}_{\Omega}^2 = \sum_{T\in\T_H} \underset{=:\mbox{I}}{\underbrace{( {\vec A}^{-1} (\vec w^T-\vec w^{T,k}), (1- \eta_{T,k+1})\vec z )}} + \underset{=:\mbox{II}}{\underbrace{({\vec A}^{-1} (\vec w^T-\vec w^{T,k}),  \eta_{T,k+1} \vec z )}}.
\end{equation*}
The first term is estimated by
\begin{align*}
\mbox{I} &\le \vertiii{\vec w^T-\vec w^{T,k}}_{\Omega} \vertiii{\vec z (1- \eta_{T,k+1})}_{U_{k+1}(T)}
\le \vertiii{\vec w^T-\vec w^{T,k}}_{\Omega} \vertiii{\vec z}_{U_{k+1}(T)}.
\end{align*}
For the second term we have $\mathbf{z} \in \Khf$, hence there exists again a
skew-symmetric matrix $\psi=\psi(\mathbf{z})$ with the properties as
in Lemma~\ref{existence-of-skew-symmetric-matrix} with
\begin{align*}
\eta_{T,k+1} \mathbf{z} =\eta_{T,k+1} \nabla \cdot \psi = \nabla \cdot ( \eta_{T,k+1} \psi ) - \psi \nabla \eta_{T,k+1}.
\end{align*}
We define $\tilde{\vec z}:=\Ih (\nabla \cdot (\eta_{T,k+1} \psi)) - (\IH \circ \Ih) \left( \nabla \cdot (\eta_{T,k+1} \psi) \right)$. Using Lemma~\ref{trunc-lemma-1} and supp$(\eta_{T,k+1} \mathbf{z})\cap$supp$(\vec w^{T,k})=\emptyset$ we get 
\begin{equation*}
\begin{aligned}
\lefteqn{({\vec A}^{-1} (\vec w^T - \vec w^{T,k}), \eta_{T,k+1} \mathbf{z}  ) = ({\vec A}^{-1} \vec w^T , \eta_{T,k+1} \mathbf{z}  )}\\
&\quad \overset{\eqref{trunc-lemma-2-proof-step-1}}{=} \int_{\Omega\setminus U_{k}(T)} {\vec A}^{-1}\vec w^T \cdot \left( \nabla \cdot (\eta_{T,k+1} \psi ) - \psi \nabla \eta_{T,k+1} - \tilde{\mathbf{z}} \right)\\
&\quad \overset{\phantom{(99)}}{=} \int_{\Omega\setminus U_{k}(T)} {\vec A}^{-1}\left( \vec w^T - \vec w^{T,k} \right) \cdot \left( \nabla \cdot (\eta_{T,k+1} \psi ) - \psi \nabla \eta_{T,k+1} - \tilde{\mathbf{z}} \right). \\
\end{aligned}
\end{equation*}
Now proceed as in Lemma \ref{trunc-lemma-2} to obtain
\begin{equation*}
  \mbox{II}\lesssim \lambda(H/h)\vertiii{\vec w^T-\vec w^{T,k}}_\Omega \vertiii{\vec z}_{U_{k+1}(T)}.
\end{equation*}
Combining the estimates for I and II and applying Hölder's inequality finally yields, for $k \ge 1$,
\begin{equation}
\begin{aligned}
\label{lemma-a-3-proof-eq-0}\vertiii{\vec z}_{\Omega}^2
&\lesssim \lambda(H/h)\sum_{T\in\T_H} \vertiii{\vec w^T-\vec w^{T,k}}_\Omega \vertiii{\vec z}_{U_{k+1}(T)}\\
&\lesssim k^{\frac{d}{2}} \lambda(H/h) \left( \sum_{T\in\T_H} \vertiii{\vec w^T-\vec w^{T,k}}_{\Omega}^2 \right)^{\frac{1}{2}} \vertiii{\vec z}_{\Omega}.
\end{aligned}
\end{equation}
It remains to bound $\vertiii{\vec w^T-\vec w^{T,k}}_{\Omega}^2$. In order to do this, we use Galerkin orthogonality for the  local problems, which gives us
\begin{align*}
\label{galerkin-orthogonality-local-eq} \vertiii{\vec w^T-\vec w^{T,k}}_{\Omega}^2 \le \inf_{\tilde{\vec w}^{T,k} \in \oKhf{U_k(T)}} \vertiii{\vec w^T-\tilde{\vec w}^{T,k}}_{\Omega}^2.
\end{align*}
Again, we use Lemma \ref{trunc-lemma-2} to show
\begin{eqnarray}
\vertiii{\vec w^T-\vec w^{T,k}}_{\Omega}^2 \lesssim \theta^{2k/\lambda(H/h)} \vertiii{\vec w^T }_{\Omega}^2.
\label{lemma-a-3-proof-eq-3}
\end{eqnarray}
Combining \eqref{lemma-a-3-proof-eq-0} and \eqref{lemma-a-3-proof-eq-3} proves the lemma.
\end{proof}

\subsection{Inf-sup stability revisited}
\label{sec:infsup}
The decay results can be used to prove another inf-sup stability
constant $\gamma_k^1$ in addition to $\gamma_k^0$ from
Lemma~\ref{lem:uniquelocalized} for the bilinear form $b(\cdot,\cdot)$ with the
localized multiscale space.  Using Lemma~\ref{trunc-lemma-3}, we
obtain
\begin{equation*}
  \begin{aligned}
    \| G_{h,k} \vec v - G_h \vec v\|_{L^{2}(\Omega)}^2 & = \left\| \sum_{T\in\triH} (G^T_{h,k} \vec v - G^T_h \vec v) \right\|_{L^{2}(\Omega)}^2 \\
    & \lesssim k^d \lambda(H/h)^2 \theta^{2k/\lambda(H/h)} \sum_{T \in \triH} \| G^T_h \vec v\|_{L^2(\Omega)}^2 \\
    & \lesssim k^d \lambda(H/h)^2 \theta^{2k/\lambda(H/h)} \| \vec v\|_{L^2(\Omega)}^2. \\
  \end{aligned}
\end{equation*}
We get the following stability 
\begin{equation*}
  \begin{aligned}
    \| G_{h,k} \vec v \|_{L^{2}(\Omega)} & \le \| G_{h,k} \vec v - G_h \vec v\|_{L^{2}(\Omega)} + \| G_h \vec v\|_{L^{2}(\Omega)} \\
    & \lesssim (k^{d/2} \lambda(H/h) \theta^{k/\lambda(H/h)} + 1) \| \vec v \|_{L^{2}(\Omega)}. \\
  \end{aligned}
\end{equation*}
Using the same technique as in Lemma~\ref{lem:uniquelocalized}, we
obtain an inf-sup stability constant $\gamma_k^1 :=
\gamma(2+k^{d/2} \lambda(H/h) \theta^{k/\lambda(H/h)})^{-1}$.

For the nodal Raviart--Thomas interpolation operator $\Pi_H$,
$\lambda(H/h)$ depends on $h$ and $H$, and we cannot obtain a uniform
bound on the constant for this estimate either. However, for
$L^2$-stable Cl\'ement-type interpolation operators (discussed in
Remark~\ref{rem:clement}), we have $\lambda(H/h) \equiv 1$,
independently of $h$ and $H$. If using such an interpolator in place
of $\Pi_H$, the inf-sup stability constant $\gamma_k^1$ can be bounded
from below by a positive constant independent of $h$ and $H$, since
$k^{d/2} \theta^{k}$ is bounded from above with respect to $k$.

\section{Numerical experiments}
\label{sec:experiments}
Four numerical experiments are presented in this section. Their
purpose is to show that the error estimate for the localized
multiscale method presented in Theorem~\ref{thm:errorestimate} is
valid and useful for determining the patch sizes and that the method
is competitive.

A brief overview of the implementation of the method follows. The two
dimensional Raviart--Thomas finite element is used. For all free
degrees of freedom $e$ (interior edges), the localized global
corrector $G_{h,k}\Phi_e$ for the corresponding basis function
$\Phi_e$ is computed according to equation
\eqref{eq:loccorrector}. The additional constraints on the test and
trial functions to be in the kernel to the coarse Raviart--Thomas
projection operator are implemented using Lagrange multipliers (in
addition to those already there due to the mixed formulation). The
corrector problems are cheap since they are solved only on small
patches. This can be done in parallel over all basis
functions. Finally, problem \eqref{eq:locmixedbilinear} is
solved. Regarding the linear system arising here, we compare it with
the linear system arising from a standard Raviart--Thomas
discretization (using $V_H$ for the flux) of the mixed formulation on
the corase mesh:
\begin{equation*}
  \begin{pmatrix}
    \vec K & \vec B^T \\
    \vec B & \vec 0 \\
  \end{pmatrix} = 
  \begin{pmatrix} \vec 0 \\ \vec b \\
  \end{pmatrix},
\end{equation*}
for matrices $\vec K$ and $\vec B$ and a vector $\vec b$. The
difference with the multiscale method is that matrix corresponding to
the bilinear form $a(\cdot,\cdot)$ is computed using the low
dimensional modified localized multiscale basis $\{\Phi_E -
G_{h,k}\Phi_E\}_E$ spanning $\VmsHk$. Since the correctors are
divergence free, $\vec K$ is replaced by a different matrix $\vec
{\tilde K}$ in the system above, whereas $\vec B$ and $\vec b$ are left intact.

In all numerical experiments below, the diffusion matrix is diagonal
with identical diagonal elements, $\vec A(x) = A(x)\vec I$, with $\vec
I$ being the identity matrix, for a scalar-valued function $A$. 

\subsection{Investigation of error from localization}
\label{sec:firstexperiment}
In this experiment, we investigate how the error in energy norm of the
localized multiscale solution is affected by the localization to
patches of the correctors. The error due to localization is bounded by
the second term in the estimate in
Theorem~\ref{thm:errorestimate}. This term will be the focus of this
experiment.

The computational domain is the unit square $\Omega = [0,1]^2$ and the
source function is given by
\begin{equation*}
  f(x) = \begin{cases} 1 & \text{if }x \in [0,1/4]^2,\\
    -1 & \text{if } x \in [3/4,1]^2,\\
    0 & \text{otherwise}.
  \end{cases}
\end{equation*}
We consider three different diffusion coefficients $A$:
\begin{enumerate}
  \item Constant: $A(x) = 1$ in the whole domain.
  \item Noise: $A(x)$ is piecewise constant on a $2^{7} \times 2^{7}$
    uniform rectangular grid. In each grid cell, the value of $A$ is
    equal to a realization of $\exp(10\omega)$, where $\omega$ is a
    cell-specific standard uniformly distributed variable.
  \item Channels: $A(x)$ is as is shown in
    Figure~\ref{fig:structured_a}. It is piecewise constant on a
    $2^{7} \times 2^{7}$ uniform rectangular grid. The coefficient
    $A(x) = 1$ for $x$ in black cells and $A(x) = \exp(10)$ for $x$ in
    white cells.
    \begin{figure}[htb]
      \centering
      \includegraphics[width=0.3\textwidth,frame]{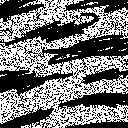}
      \caption{Coefficient $A$ defined on a $2^7 \times 2^7$ grid of
        $\Omega$.}
      \label{fig:structured_a}
    \end{figure}
\end{enumerate}
\begin{figure}[htb]
  \centering
  \begin{subfigure}{.4\textwidth}
    \centering
    \includegraphics[width=0.3\textwidth,trim=2cm 2cm 2cm 2cm,clip=true]{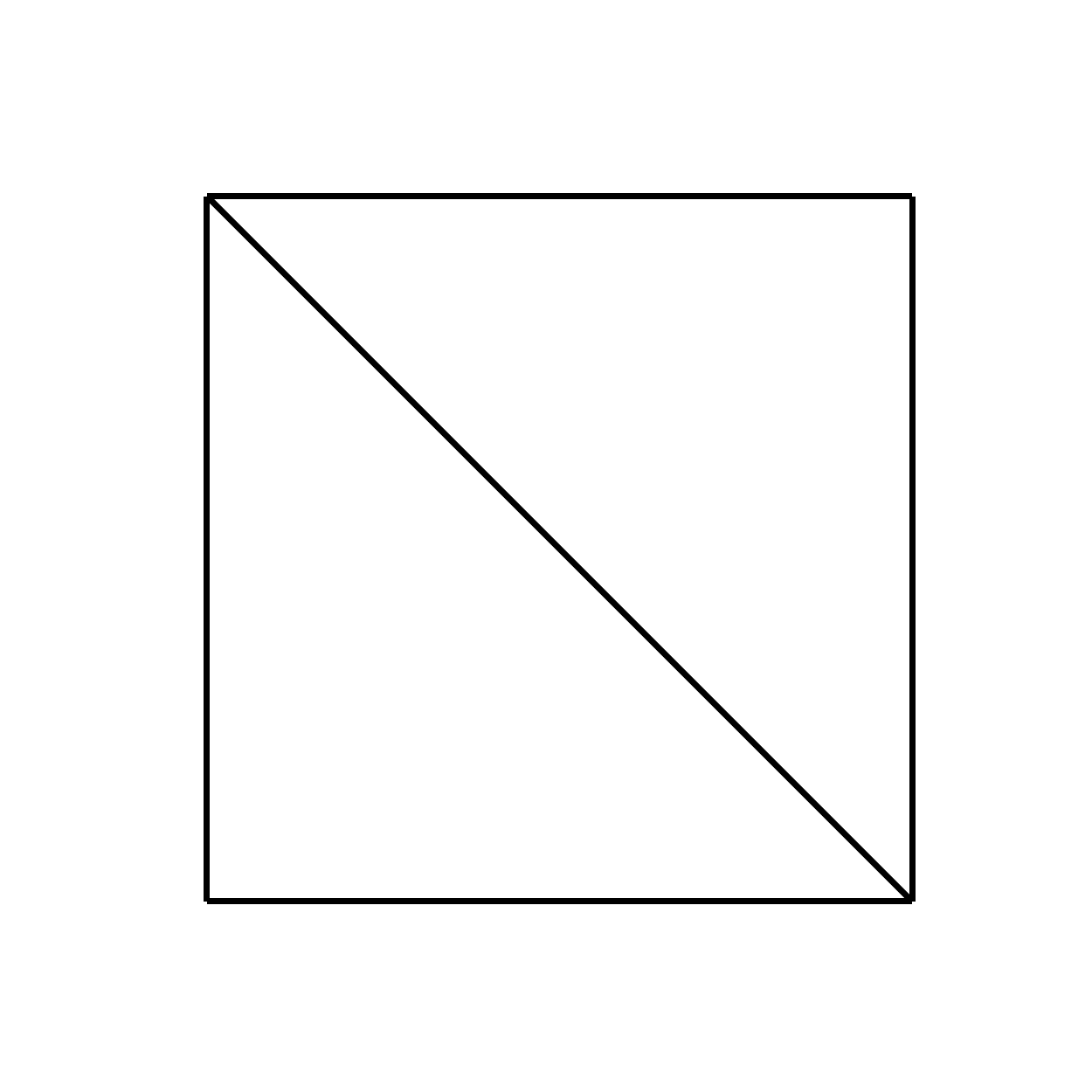}
    \caption{Coarsest mesh, $h = 1$.}
  \end{subfigure}
  \hspace{1em}
  \begin{subfigure}{.4\textwidth}
    \centering
    \includegraphics[width=0.3\textwidth,trim=2cm 2cm 2cm 2cm,clip=true]{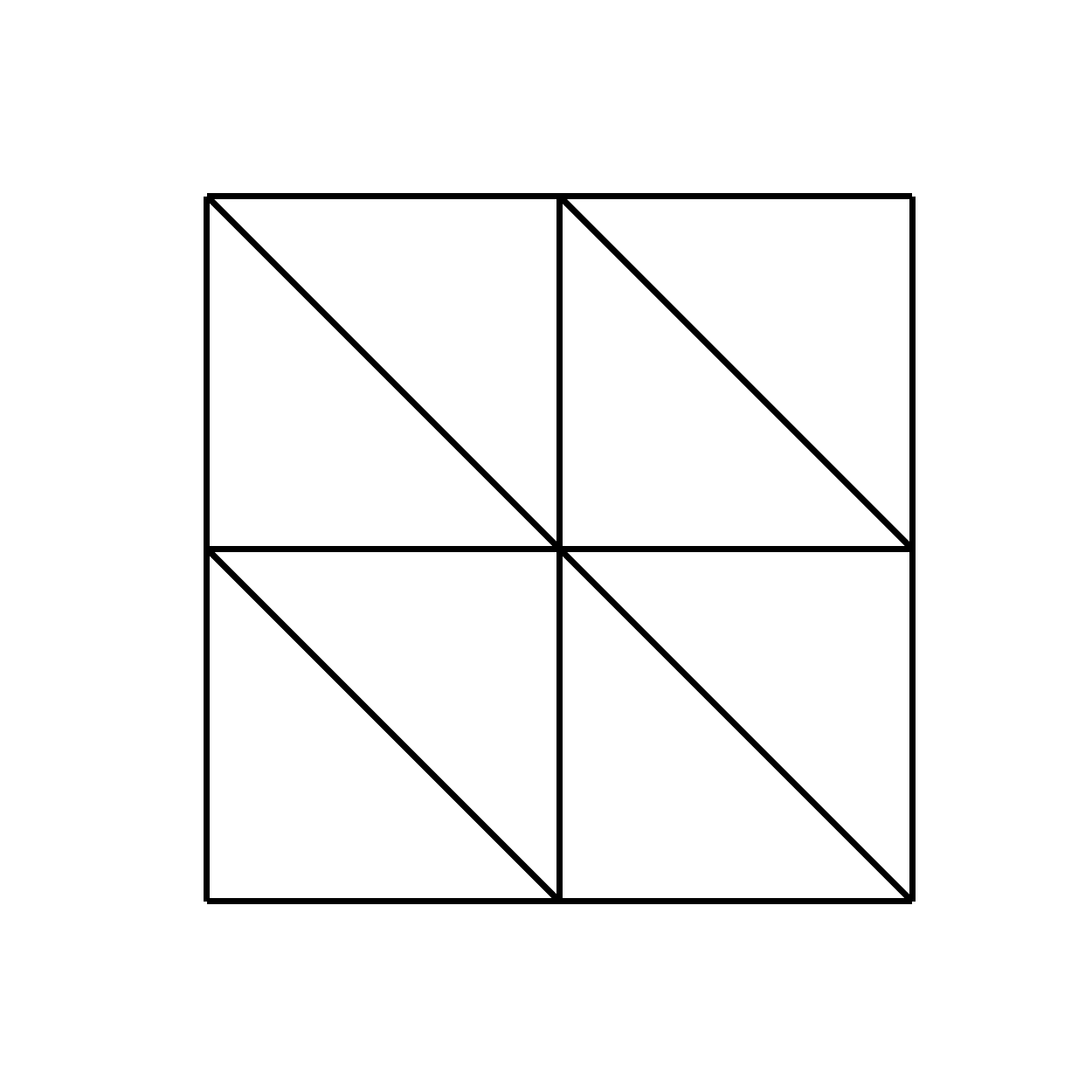}
    \caption{One refinement, $h = 1/2$.}
  \end{subfigure}
  \caption{Family of triangulations of the unit square.}
  \label{fig:mesh}
\end{figure}
Figure~\ref{fig:mesh} shows the mesh used in the experiment. Both fine
and coarse meshes are constructed as shown in the figure. A reference
solution $\vec u_h$ was computed with the standard Raviart--Thomas
spaces $V_h$ and $Q_h$ with $h = 2^{-8}$. Solutions $\umsHk$ to the
localized multiscale problem were computed using $H = 2^{-2}, 2^{-3},
\ldots, 2^{-6}$. The patch size $k$ was chosen as
\begin{equation*}
  k = C (1+\log_2(H/h))^{1/2}\log_2(H^{-1})
\end{equation*}
rounded to the nearest integer with $C = 0.25$ and $C = 0.5$. The
relative error (using the reference solution in place of the exact
solution) in energy norm, i.e.\ $\vertiii{\vec u_h -
  \umsHk}/\vertiii{\vec u_h}$ was computed. See
Figure~\ref{fig:convergence} for the resulting convergence of this
error with respect to $H$ for the two values of $C$. Note that since
$f \in Q_H$ for all examples, the first term in \eqref{final-estimate}
vanishes. The error is hence bounded by $k^{d/2} \lambda(H/h)^2
\theta^{k/\lambda(H/h)} \|f\|_{L^2(\Omega)}$, which allows for a
careful investigation of the influence of $k$, $H$ and $h$. A
reference line proportional to $H^2$ is plotted for guidance. We can
see that we achieve convergence for both choices of $C$. However,
since $k$ is rounded to an integer, the convergence plots have a
staggered appearance. This example shows that the error due to
localization can be kept small and decreases with $H$. The plots also
show the relative error in energy norm for the standard
Raviart--Thomas discretization on the coarse mesh. It is evident that
the localized multiscale space has good approximation properties since
it permits convergence while the standard space of the same dimension does not.
\begin{figure}[p]
  \centering
  \begin{subfigure}{\textwidth}
    \centering
    \includegraphics[trim=0 0.3cm 0 0,clip=true]{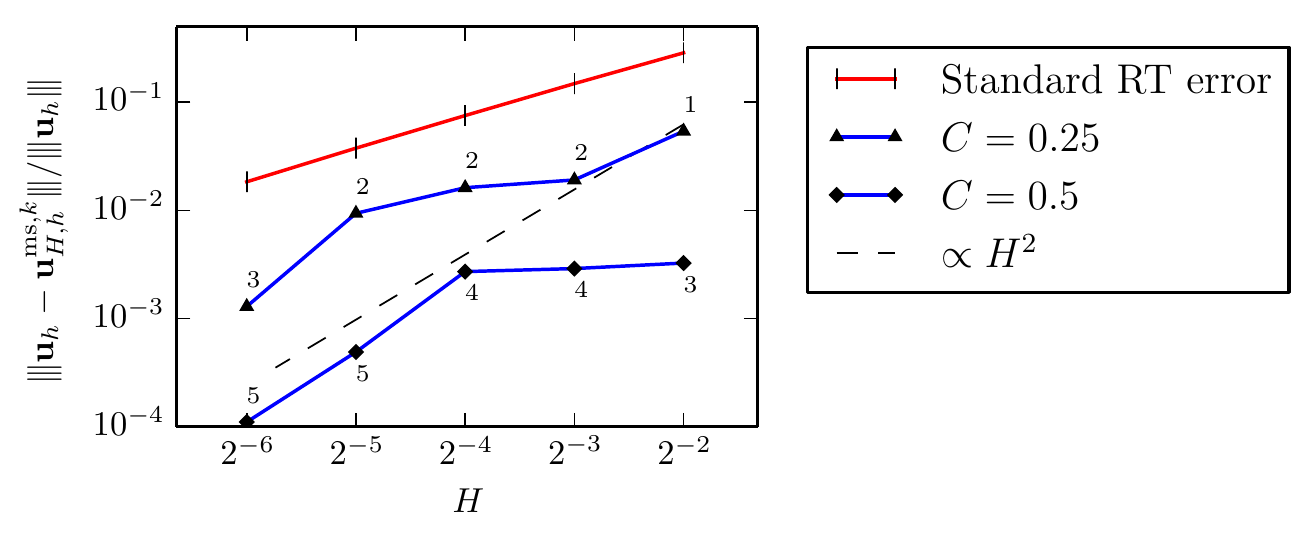}
    \caption{Diffusion coefficient is constant.}
  \end{subfigure}
  \begin{subfigure}{\textwidth}
    \centering
    \includegraphics[trim=0 0.3cm 0 0,clip=true]{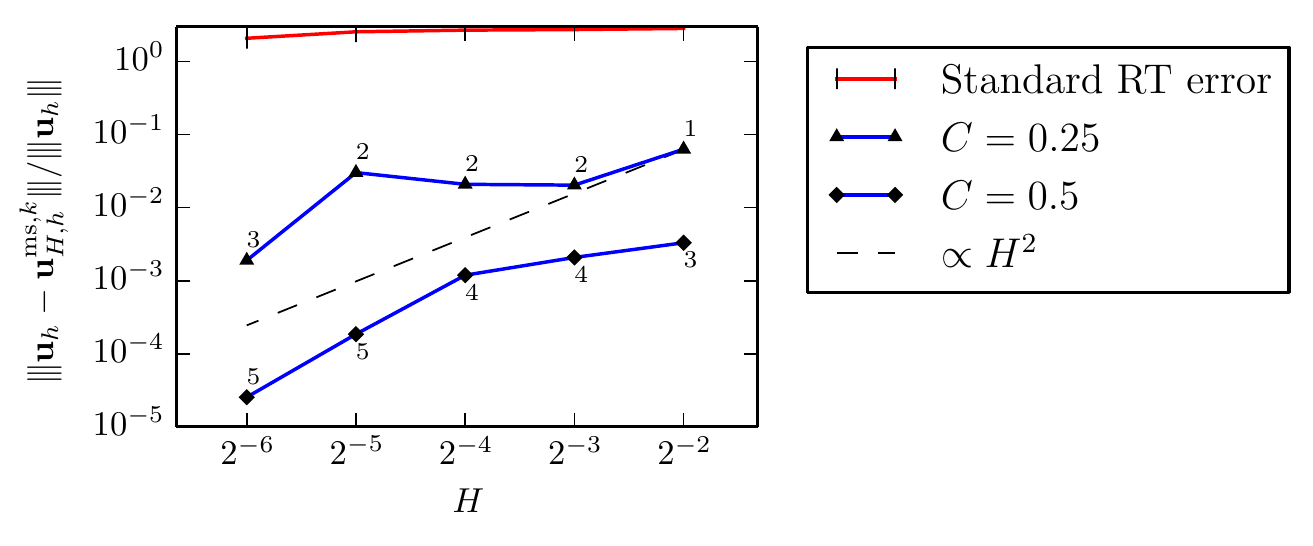}
    \caption{Diffusion coefficient is noisy.}
  \end{subfigure}
  \begin{subfigure}{\textwidth}
    \centering
    \includegraphics[trim=0 0.3cm 0 0,clip=true]{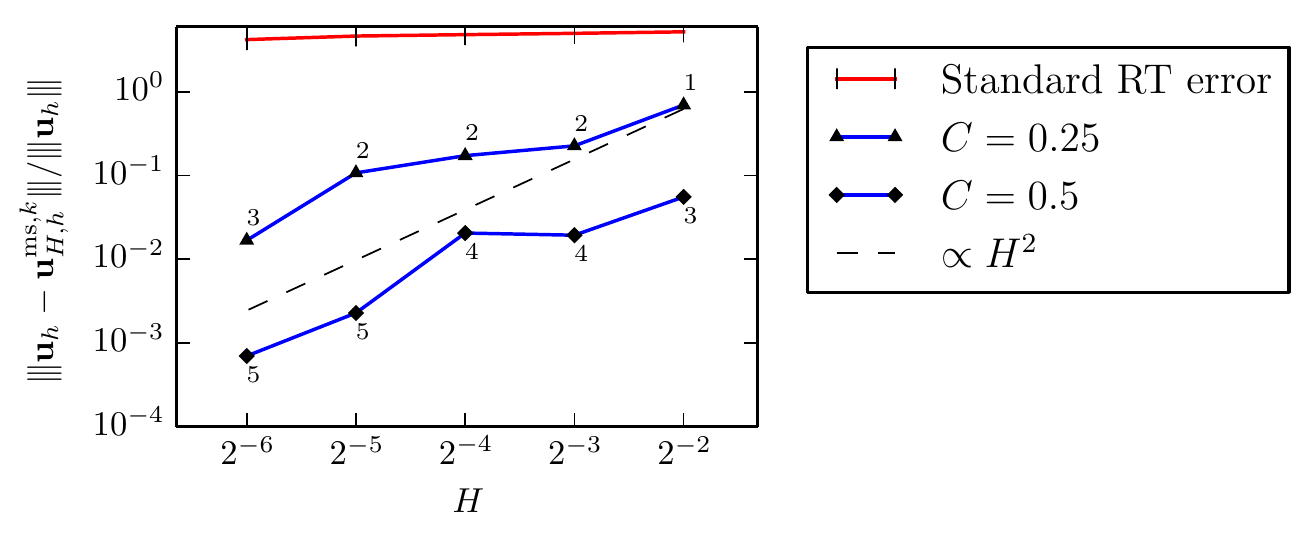}
    \caption{Diffusion coefficient has channel structures.}
  \end{subfigure}
  \caption{Convergence plots for localization error
    experiments. Relative error in energy norm for three choices of
    $A$, for different values of the constant $C$ determining the
    patch size. The number adjacent to a point is the actual value of
    $k$ for the specific simulation corresponding to that point.}
  \label{fig:convergence}
\end{figure}

\subsection{Investigation of instability}
\label{sec:instability}
In this experiment we show how singularity-like features can appear in
the solution, probably as a result of high contrast in combination
with the $L^2$-instability of the nodal Raviart--Thomas interpolator.

Again, we consider the unit square $\Omega = [0,1]^2$. The diffusion
coefficient $A$ is chosen according to Figure~\ref{fig:singularity_a}.
\begin{figure}[htb]
  \centering
  \includegraphics[width=0.3\textwidth,frame]{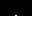}
  \caption{Coefficient $A$ defined on a $2^5 \times 2^5$ grid of
    $\Omega$ where $A(x) = 1$ for $x$ in black cells and $A(x) =
    \exp(10)$ for $x$ in white cells.}
  \label{fig:singularity_a}
\end{figure}
In other words,
$A$ is defined as
\begin{equation*}
  A(x) = \begin{cases} \exp(10) & \text{if }x_2 < 1/2 \text{ or } x \in [\frac{1}{2}-2^{-5}, \frac{1}{2}+2^{-5}]\times [\frac{1}{2}, \frac{1}{2}+2^{-5}],\\
    1 & \text{otherwise}.
  \end{cases}
\end{equation*}
The source function is chosen as
\begin{equation*}
  f(x) = \begin{cases} -1 & \text{if }x_2 < 1/2,\\
    1 & \text{otherwise}.
  \end{cases}
\end{equation*}
This particular choice of $A$ and $f$ yields a localized multiscale
solution with a clear singularity-like feature at $x = (x_1, x_2) = (1/2, 1/2)$ in the localized
multiscale solution.

We use the family of triangulations presented in Figure~\ref{fig:mesh}
and fix $H = 1/4$ so that $f$ is resolved on the coarse scale. Then $f
\in Q_H$ and all error stems from localization (see
Theorem~\ref{thm:errorestimate}). We let the resolution $h$ of the
fine space be $h = 2^{-5}, 2^{-6},\ldots, 2^{-9}$. Choosing $k = 2$,
we compute the localized multiscale solution $\umsHk$ and reference
solution $\vec u_h$ for the given values of $h$.

From the error estimate in Theorem~\ref{thm:errorestimate}, we expect
to have
\begin{equation*}
  \begin{aligned}
    \vertiii{\vec u_h - \umsHk} & \lesssim k^{d/2} \lambda(H/h)^2 \theta^{k/\lambda(H/h)} \|f\|_{L^2(\Omega)} \\
    & \propto \log\left(h^{-1}\right) \quad \text{as $h \to 0$}.\\
  \end{aligned}
\end{equation*}
The energy norm of the error is plotted in
Figure~\ref{fig:singularity_error}. We can see that for this
particular problem and range of $h$, the error increases with $h$ and
with the rate $\log(h^{-1})$ as predicted by the error estimate. However,
the error estimate seems not to be sharp for this particular
example. Figure~\ref{fig:singularity_solutions} shows the reference
and multiscale flux solutions. The magnitude of the reference solution
is in the range $[0,3]$, while the multiscale solution has a spike
reaching magnitude $30$ at $x = (1/2, 1/2)$. Interesting to note is
that the singularities vanish for the ideal multiscale method, i.e.\
without localization, see Lemma~\ref{lem:idealerrorestimate}.
\begin{figure}[htb]
  \centering
  \includegraphics[]{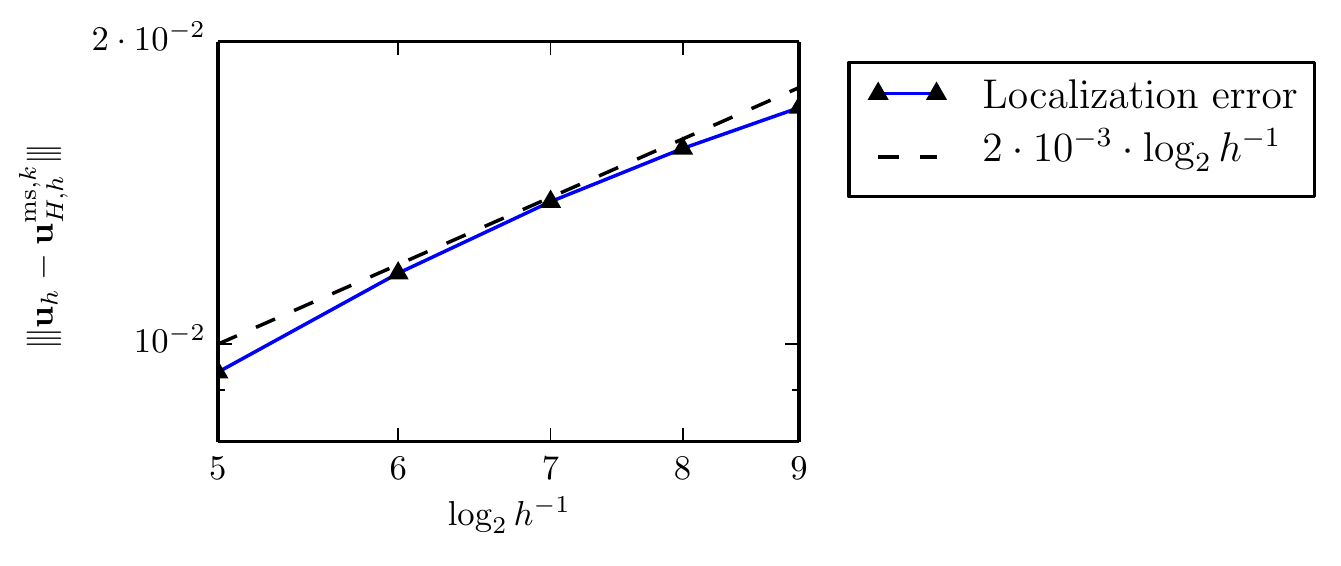}
  \caption{Divergence of the energy norm of the localization error of
    a particular multiscale solution as $h$ decreases.}
  \label{fig:singularity_error}
\end{figure}
\begin{figure}[htb]
  \captionsetup[subfigure]{justification=raggedright}
  \centering
  \begin{subfigure}{.47\textwidth}
    \includegraphics[width=\textwidth,trim=2.5cm 1.3cm 2.5cm 1.3cm,clip=true]{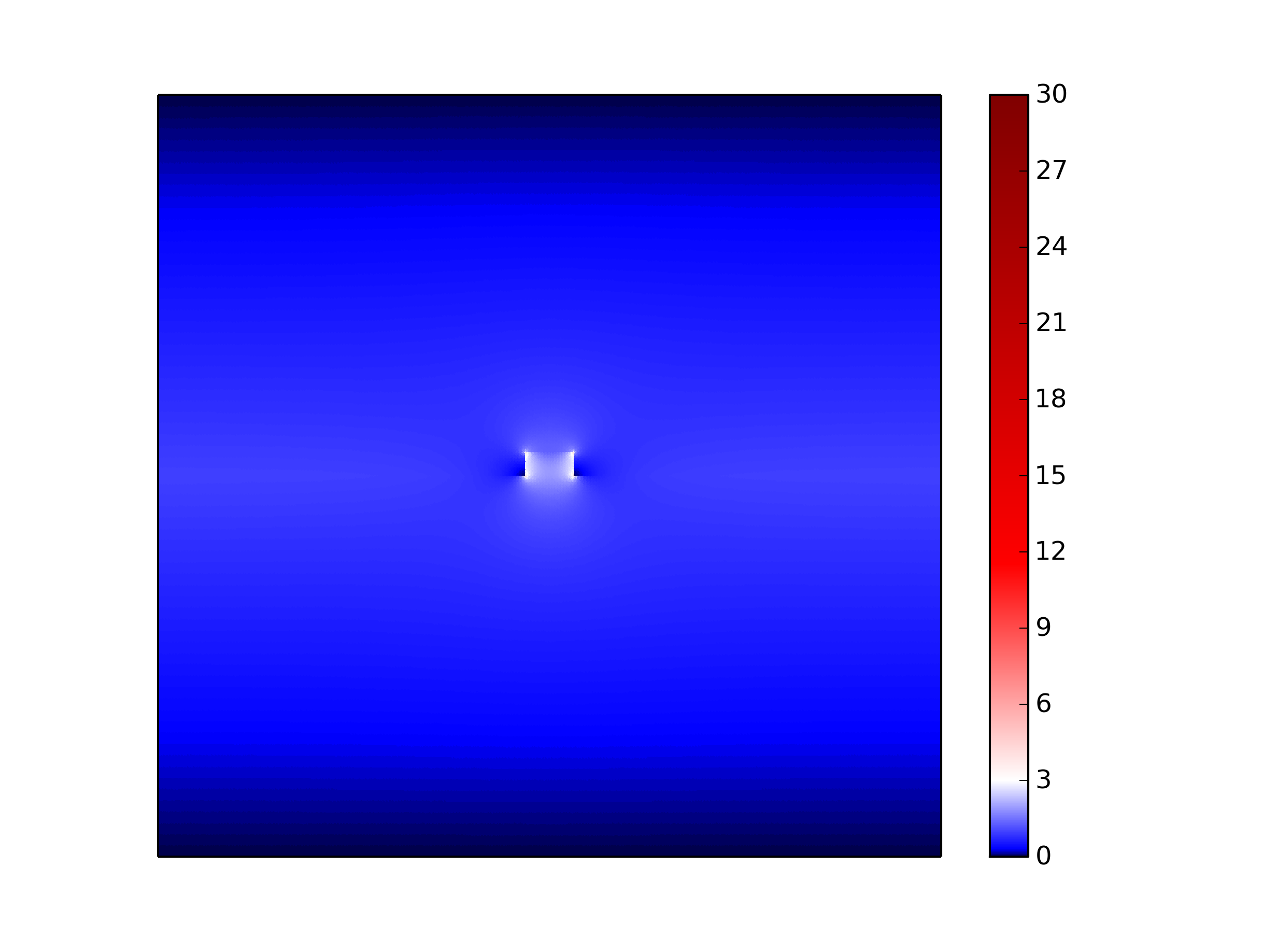}
    \caption{Reference solution, $h = 2^{-9}$. \phantom{text to make it two lines}}
  \end{subfigure}
  \hspace{1em}
  \begin{subfigure}{.47\textwidth}
    \includegraphics[width=\textwidth,trim=2.5cm 1.3cm 2.5cm 1.3cm,clip=true]{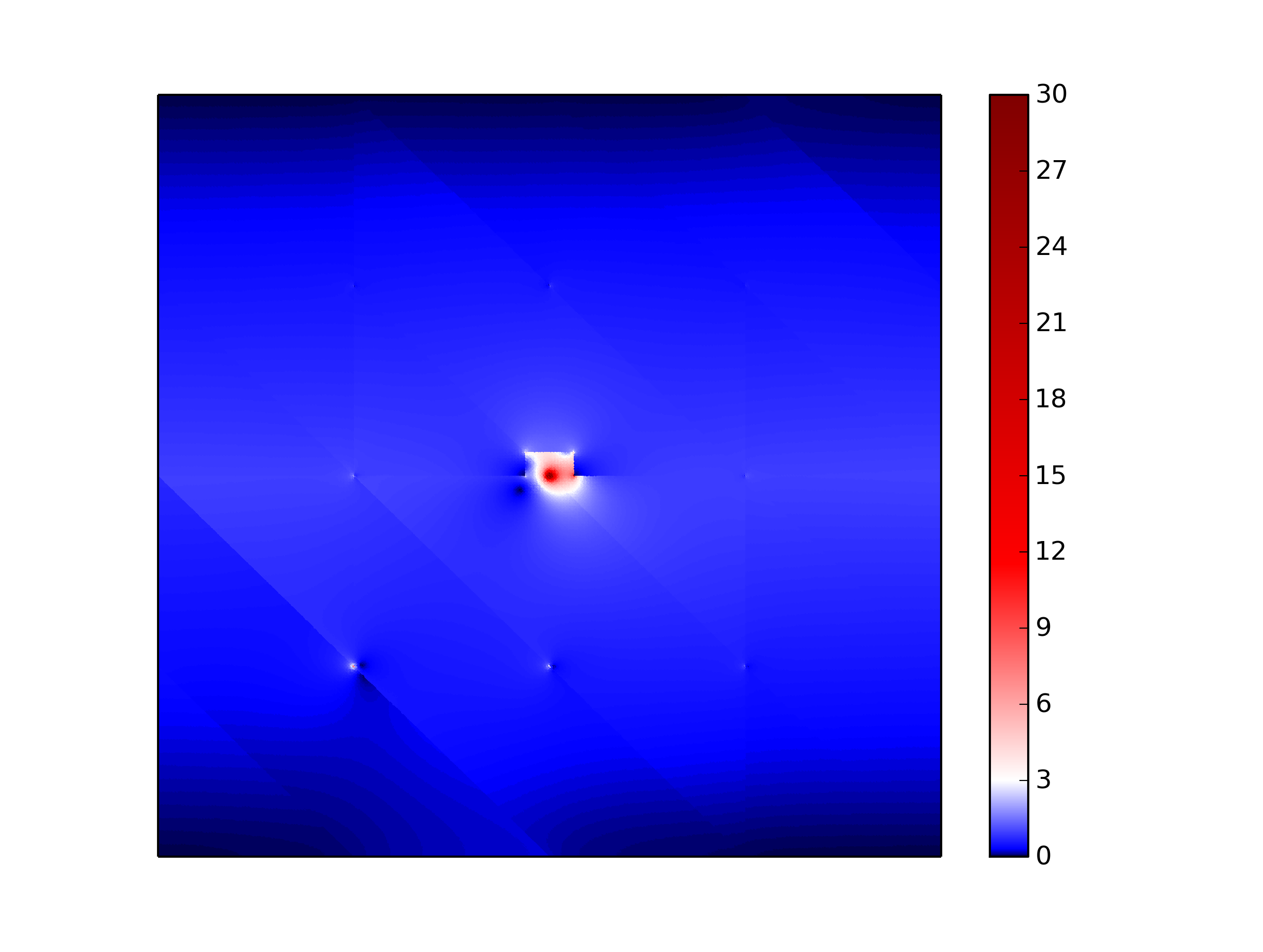}
    \caption{Multiscale solution, $h = 2^{-9}$, $H = 2^{-2}$, $k = 2$.}
  \end{subfigure}
  \caption{Magnitude of flux at the centroid of the triangles.}
  \label{fig:singularity_solutions}
\end{figure}

\subsection{Convergence in an L-shaped domain}
Next, we consider an L-shaped domain with noisy diffusion coefficient
$A$ (case 2.\ in Section~\ref{sec:firstexperiment}) and with
$f \notin Q_H$. In this experiment, we show that the localization
error investigated in the previous section can be dominated by errors
from projecting $f$.

We use the domain $\Omega = [0,1]^2 \setminus [1/2,1]\times[0,1/2]$
and the triangulation presented in Figure~\ref{fig:lshape}.  Both fine
and coarse meshes are constructed as shown in the figure.
\begin{figure}[htb]
  \centering
  \begin{subfigure}{.4\textwidth}
    \centering
    \includegraphics[width=0.3\textwidth,trim=2cm 2cm 2cm 2cm,clip=true]{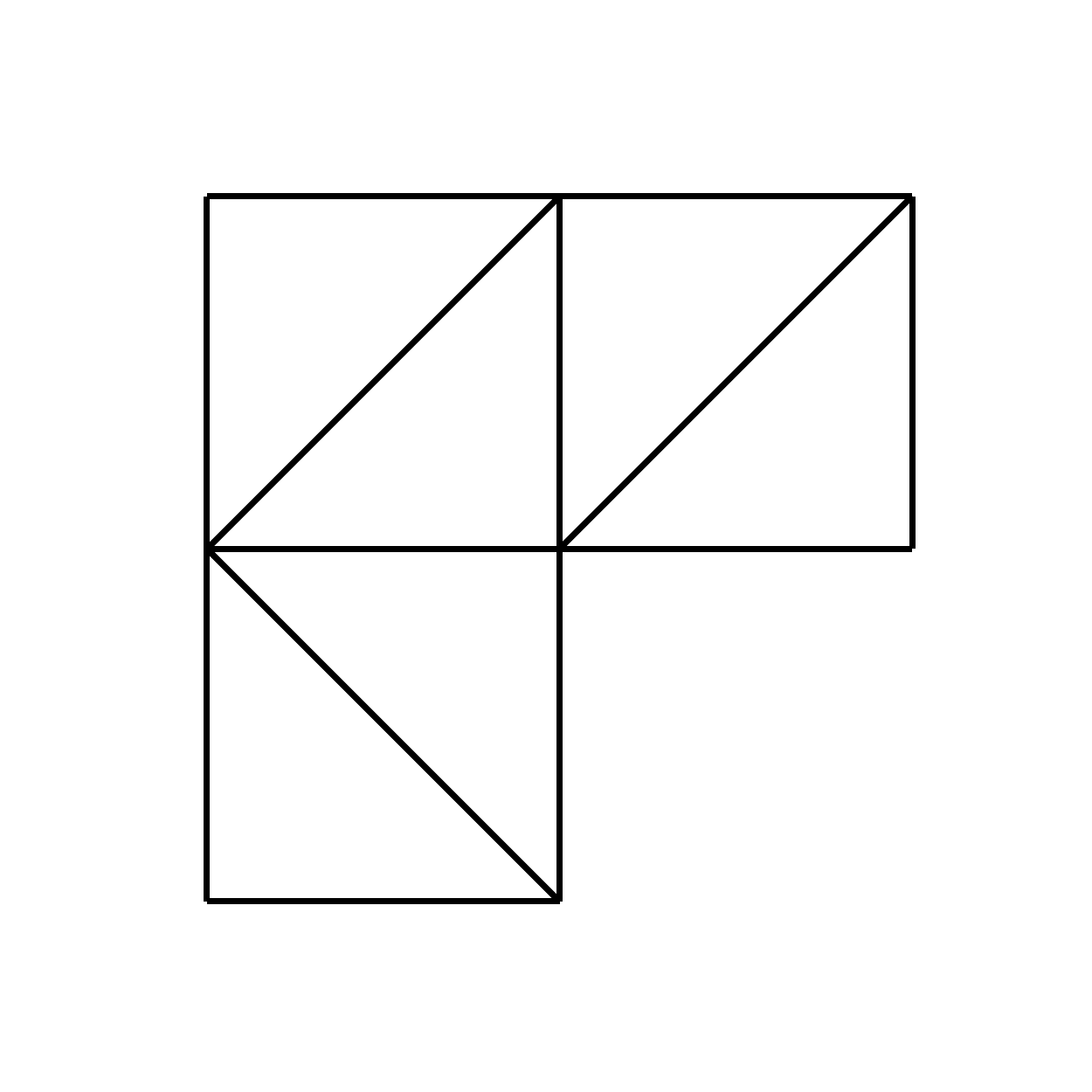}
    \caption{Coarsest mesh, $h = 1/2$.}
  \end{subfigure}
  \hspace{1em}
  \begin{subfigure}{.4\textwidth}
    \centering
    \includegraphics[width=0.3\textwidth,trim=2cm 2cm 2cm 2cm,clip=true]{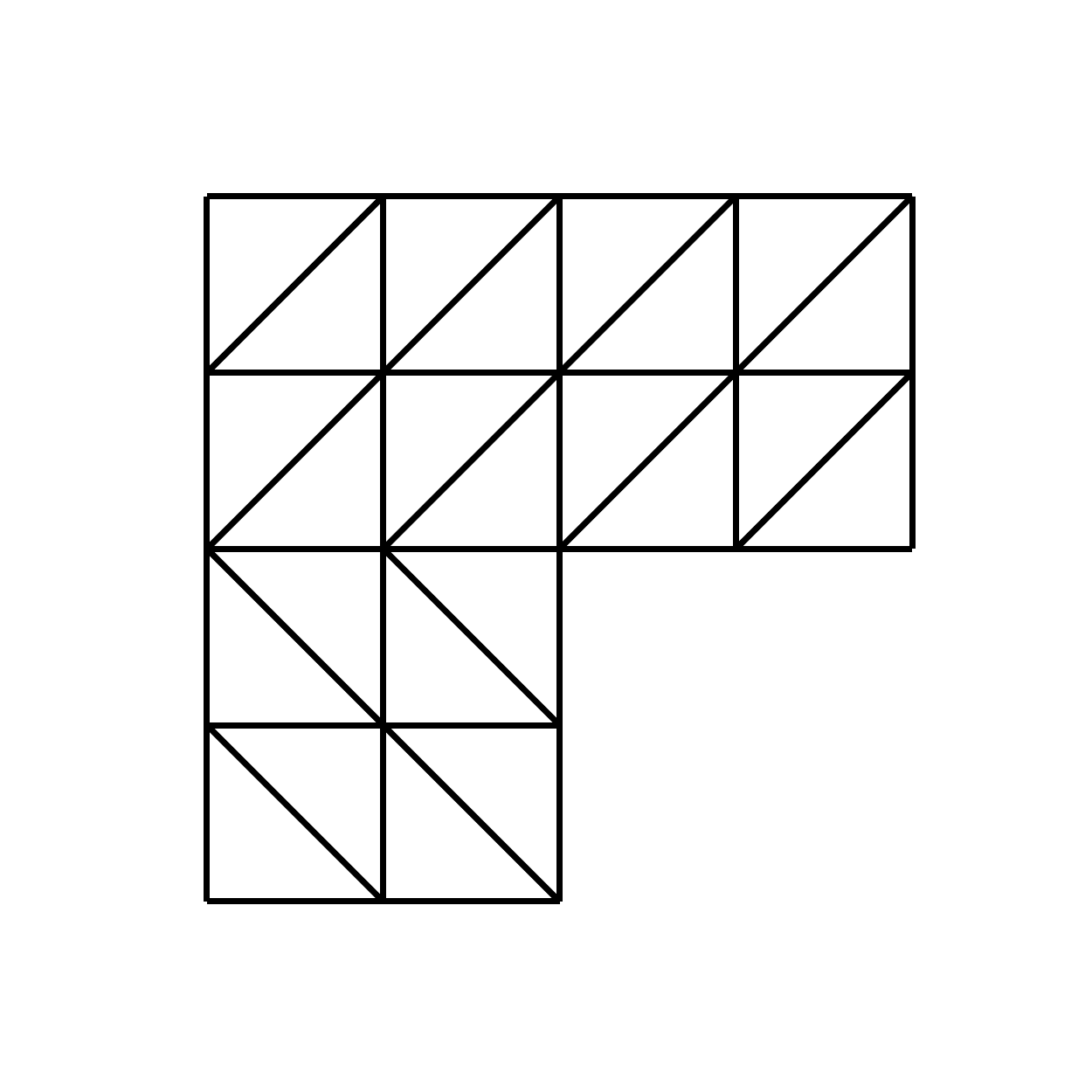}
    \caption{One refinement, $h = 1/4$.}
  \end{subfigure}
  \caption{Family of triangulations of the L-shaped domain.}
  \label{fig:lshape}
\end{figure}
Further, we choose the source function as
\begin{equation*}
  f(x) = \begin{cases} 1/2 + x_1 - x_2 & \text{if }x_2 < 1/2,\\
    -(1/2 + x_1 - x_2) & \text{if }x_1 > 1/2,\\
    0 & \text{otherwise}.
  \end{cases}
\end{equation*}
Note that $f \notin Q_H$ and $\| f - P_H f \|_{L^2(\Omega)} \lesssim
H$. A reference solution $\vec u_h$ was computed
with the standard Raviart--Thomas spaces $V_h$ and $Q_h$ with $h =
2^{-8}$. Solutions $\umsHk$ to the localized multiscale problem were
computed using $H = 2^{-2}, 2^{-3}, \ldots, 2^{-6}$. The patch size
$k$ was chosen as
\begin{equation*}
  k = C(1+\log_2(H/h))^{1/2}\log_2(H^{-1})
\end{equation*}
rounded to the nearest integer, with $C = 0.25$ and $C = 0.5$. The
relative error in energy norm was recorded for the solutions
corresponding to the values of $H$. The resulting convergence plot can be found in
Figure~\ref{fig:lshapeconvergence}. We expect the first term in the
error estimate,
\begin{equation}
    \vertiii{\vec u_h - \umsHk} \lesssim H\|f - P_H f\|_{L^2(\Omega)} + k^{d/2} \lambda(H/h) \theta^{k/\lambda(H/h)} \|f\|_{L^2(\Omega)}
\end{equation}
to be of order $H^2$. From the convergence plots we can see that
$C=0.25$ is not sufficient to make the localization error of at least
order $H^2$, however, $C = 0.5$ is.
\begin{figure}[htb]
  \centering
  \includegraphics[trim=0 0.3cm 0 0,clip=true]{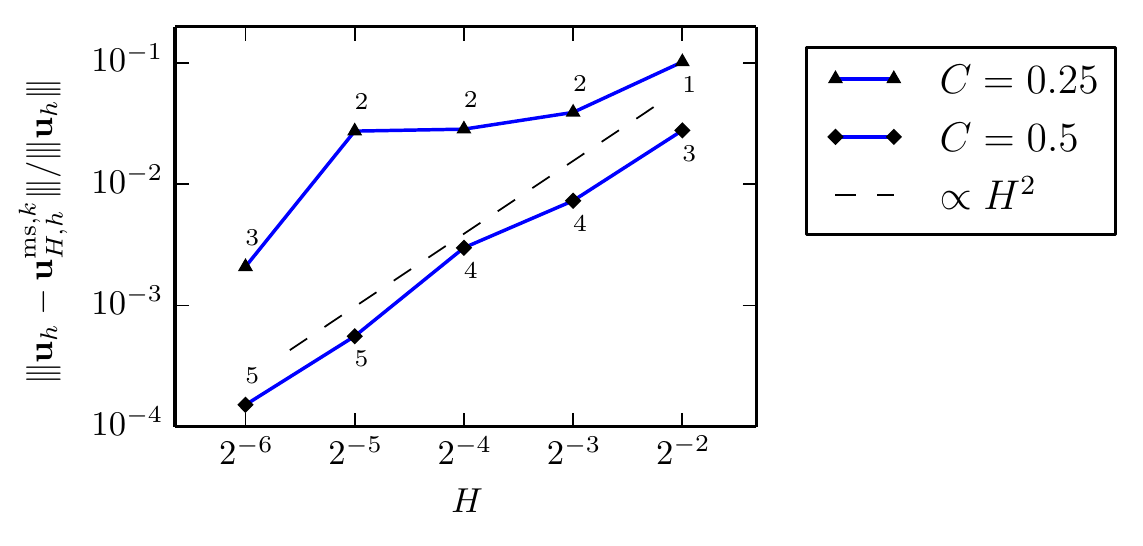}
  \caption{Convergence plot for experiment with L-shaped domain. Shows
    relative error in energy norm for two values of $C$ and a series
    of values of $H$. The number adjacent to a point is the actual
    value of $k$ for the specific simulation corresponding to that
    point.}
  \label{fig:lshapeconvergence}
\end{figure}

\subsection{Comparison with MsFEM}
We compare the proposed method with the results obtained using
the Multiscale Finite Element Method
(MsFEM)
based approach in \cite{Ar11}. The domain is $\Omega =
[0,1.2]\times[0,2.2]$ and the permeability coefficient $A$ is given in
a uniform rectangular grid of size $60 \times 220$ by the 85th
permeability layer in model 2 of SPE10 \cite{Ch01}.

The method proposed in \cite{Ar11} is based on a fine and a coarse
mesh with quadrilateral elements. The fine mesh is uniform $60 \times
220$, i.e.\ aligned with the permeability data, and the coarse mesh is
$6 \times 22$, so that each coarse element is subdivided into $10
\times 10$ fine elements. The implementation of the method proposed in
this work uses triangular meshes,
which is why we divide each of the rectangular elements 
into two triangular elements by a diagonal line drawn from the upper left corner
to the lower right corner.
As coarse mesh, we use a similar triangular
mesh that is constructed from a $6 \times 22$ rectangular mesh
such that the fine mesh is a conforming refinement of the coarse
mesh.

The (quasi-singular) source data $f$ is equal to $1$ in the 
lower left
and $-1$ in the
upper right
fine quadrilateral element. 
Note that such $f$ is a discretization 
of point sources that model production wells.
In particular, the source terms on the continuous level 
are mathematically described 
by Dirac delta functions. Hence, for $h\rightarrow 0$, we only
have $f \in W^{-m,2}(\Omega)$ for $m>\frac{d}{2}$, opposed to $f\in L^2(\Omega)$
as is required for our analysis. To account for this difference,
we follow \cite{Ma11} and compute the localized source corrections
$F_h^{T,\ell} f \in \oVhf{U_\ell(T)}$ on $\ell$-coarse-layer patches for $T
\in \triH$,
\begin{equation*}
    a(F_h^{T,\ell} f, \vec v_h^{\operatorname*{f}}) + b(\vec v_h^{\operatorname*{f}}, {\tilde F}_h^{T,\ell} f) + b(F^{T,\ell}_h f, q_h^{\operatorname*{f}}) = -(f, q_h^{\operatorname*{f}})_{T},
\end{equation*}
for all $\vec v_h^{\operatorname*{f}} \in \oVhf{U_\ell(T)}$ and
$q_h^{\operatorname*{f}} \in \oQhf{U_\ell(T)}$, where
$\oQhf{U_\ell(T)}$ is the restriction of $\Qhf$ to $U_\ell(T)$,
analogous to the definition of $\oVhf{U_\ell(T)}$. (The pressure
solution ${\tilde F}_h^{T,\ell }f$ is not needed for correcting the
flux and is discarded after its use as Lagrange multiplier). Since $f$
is non-zero only for the two triangles $T_1$ and $T_2$ in the
lower left and upper right corners,
only two such corrector problems
need to be solved. The total localized source correction is
$F_h^{\ell} f = F_h^{T_1, \ell} f + F_h^{T_2, \ell} f \in \Vhf$.

The localized corrector problems \eqref{eq:loccorrector} are
unaffected by the source correction. The right hand side of the
localized multiscale problem \eqref{eq:locmixedbilinear} is appended
with the localized source corrections and instead reads: find
$\umsHkell$ such that
\begin{equation*}
  a(\umsHkell, \vec v_h) + b(\vec v_h, \pH) + b(\umsHkell, q_H) = - (f, q_H) - a(F_h^{\ell} f, \vec v_h).
\end{equation*}
Using a value of $\ell = 0$ will be referred to as an ad-hoc source
correction, since we do not expect to have any decay of the correction
already within the source triangle itself. The source corrected
solution is $\umsHkell + F_h^{\ell} f$.

{We emphasize that the need for source correctors 
for singular source terms is not 
an exclusive drawback for our approach, but
it is a common necessity 
shared by all comparable multiscale methods in this setting. 
In particular they are also used for the MsFEM-based approach
in \cite{Ar11} that we use for our comparative study.}

The proposed localized multiscale method was used to solve for the flux
in the described problem for three corrector patch sizes: $k = 1, 2,$
and $3$. Three variants of source correction were used: i) without
source correction, i.e.\ $\umsHk$, ii) with ad-hoc source correction,
i.e.\ $\umsHkell + F_h^{\ell} f$ for $\ell = 0$ 
(without interpolation constraint), 
and iii) with source correction,
i.e.\ $\umsHkell + F_h^{\ell} f$ for $\ell = k, k+1, \infty$. A reference solution
$\vec u_h$ was computed on the fine mesh. Table~\ref{tbl:speresults}
shows the relative energy norm and $L^2$-norm of the difference
between the localized multiscale solution and the reference solution
for the different values of $k$ and $\ell$. The corresponding
$L^2$-norm of the error for the MsFEM method with oversampling HE0-OS
proposed in \cite{Ar11} is also presented in the table. Note that
HE0-OS is based on a discretization with roughly 33\% less degrees of
freedom than the proposed method, since it uses quadrilaterals instead
of triangles 
(however, since this holds for 
both the fine and the coarse mesh, the relative
change in the amount of degrees of freedom
with respect to the reference solution
is the same).
The flux solutions are plotted in
Figure~\ref{fig:speresults}.

The results show that the proposed method even without error
correction compares favorably with the homogenization based
approach. Ad-hoc error correction gives small errors for this problem
in both norms. For source correction with patch size $\ell = k$,
instabilities similar to that studied in Section~\ref{sec:instability}
cause the error to increase. However, letting $\ell = k+1$ is enough
to get errors that compare favorably with \cite{Ar11}.
\begin{table}
  \caption{Relative error in energy norm and $L^2$-norm for the SPE10-85 problem.}
  \centering
  \label{tbl:speresults}
  \begin{tabular}{lcccc}
    \toprule
    Method & $k$ & $\ell$ & Energy norm & $L^2$-norm \\
    \midrule
    \multirow{3}{10em}{Proposed method\\without source correction} 
    & 1 & $-$ & 0.7863 & 0.4069 \\
    & 2 & $-$ & 0.7856 & 0.3369 \\
    & 3 & $-$ & 0.7855 & 0.3325 \\
    \midrule
    \multirow{3}{10em}{Proposed method\\with ad-hoc source correction ($\ell = 0$)} 
    & 1 & 0 & 0.1541 & 0.2700 \\
    & 2 & 0 & 0.1515 & 0.1467 \\
    & 3 & 0 & 0.1537 & 0.1379 \\
    \midrule
    \multirow{9}{10em}{Proposed method\\with source correction ($\ell = k, k+1, \infty$)} 
    & 1 & 1 & 0.1090 & 0.8292 \\
    & 1 & 2 & 0.0459 & 0.2703 \\
    & 1 & $\infty$ & 0.0350 & 0.2504 \\
    \cmidrule(r){2-5}
    & 2 & 2 & 0.0549 & 0.7453 \\
    & 2 & 3 & 0.0185 & 0.0517 \\
    & 2 & $\infty$ & 0.0150 & 0.0490 \\
    \cmidrule(r){2-5}
    & 3 & 3 & 0.0080 & 0.0178 \\
    & 3 & 4 & 0.0051 & 0.0424 \\
    & 3 & $\infty$ & 0.0041 & 0.0088 \\
    \midrule
    HE0-OS \cite{Ar11} & $-$ & $-$ & $-$ & 0.3492\\
    \bottomrule
  \end{tabular}
\end{table}

\begin{figure}[htb]
  \centering
  \begin{subfigure}{.30\textwidth}
    \centering
    \includegraphics[width=\textwidth]{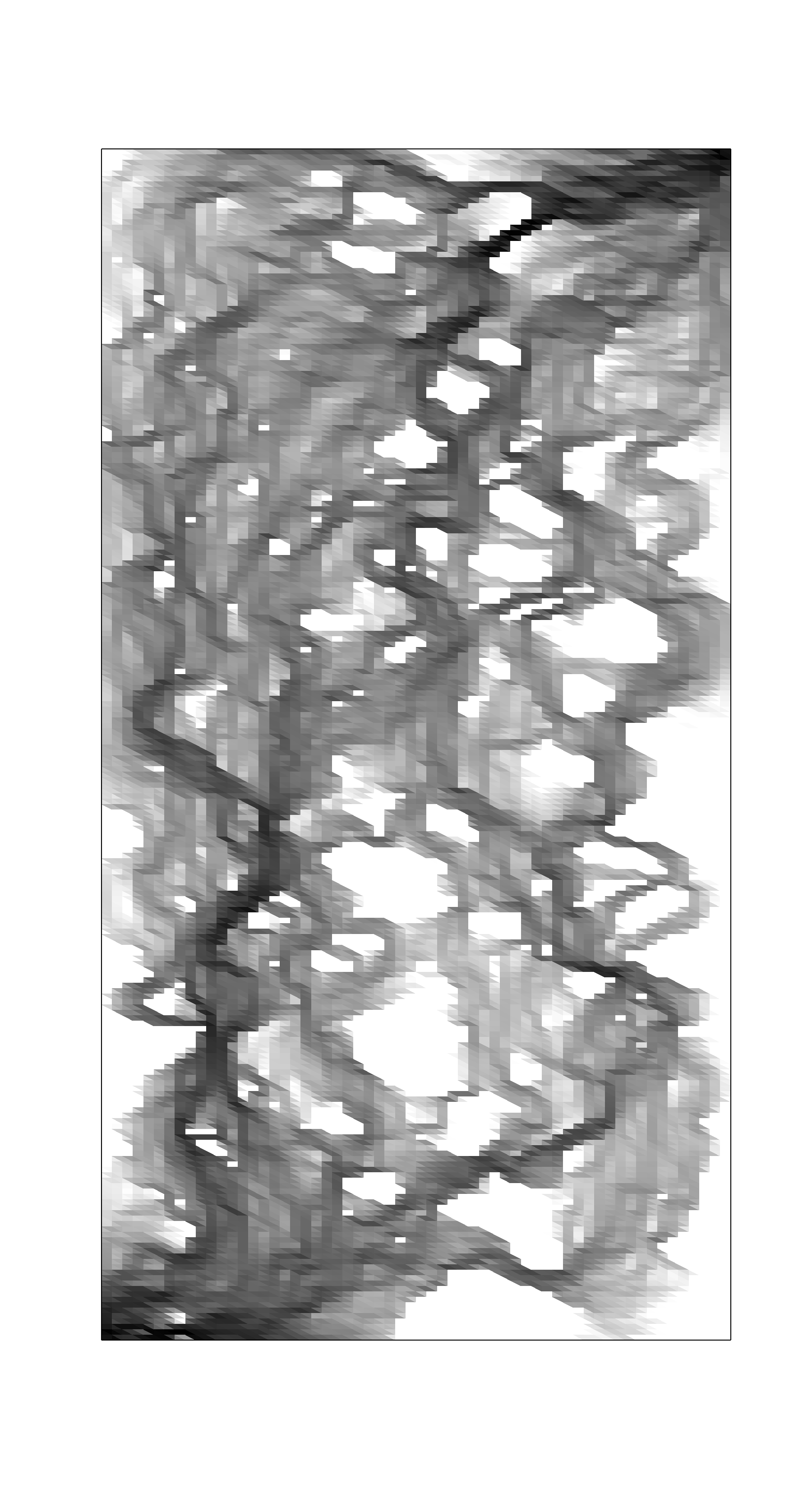}
    \caption{Reference solution}
  \end{subfigure}
  \begin{subfigure}{.30\textwidth}
    \centering
    \includegraphics[width=\textwidth]{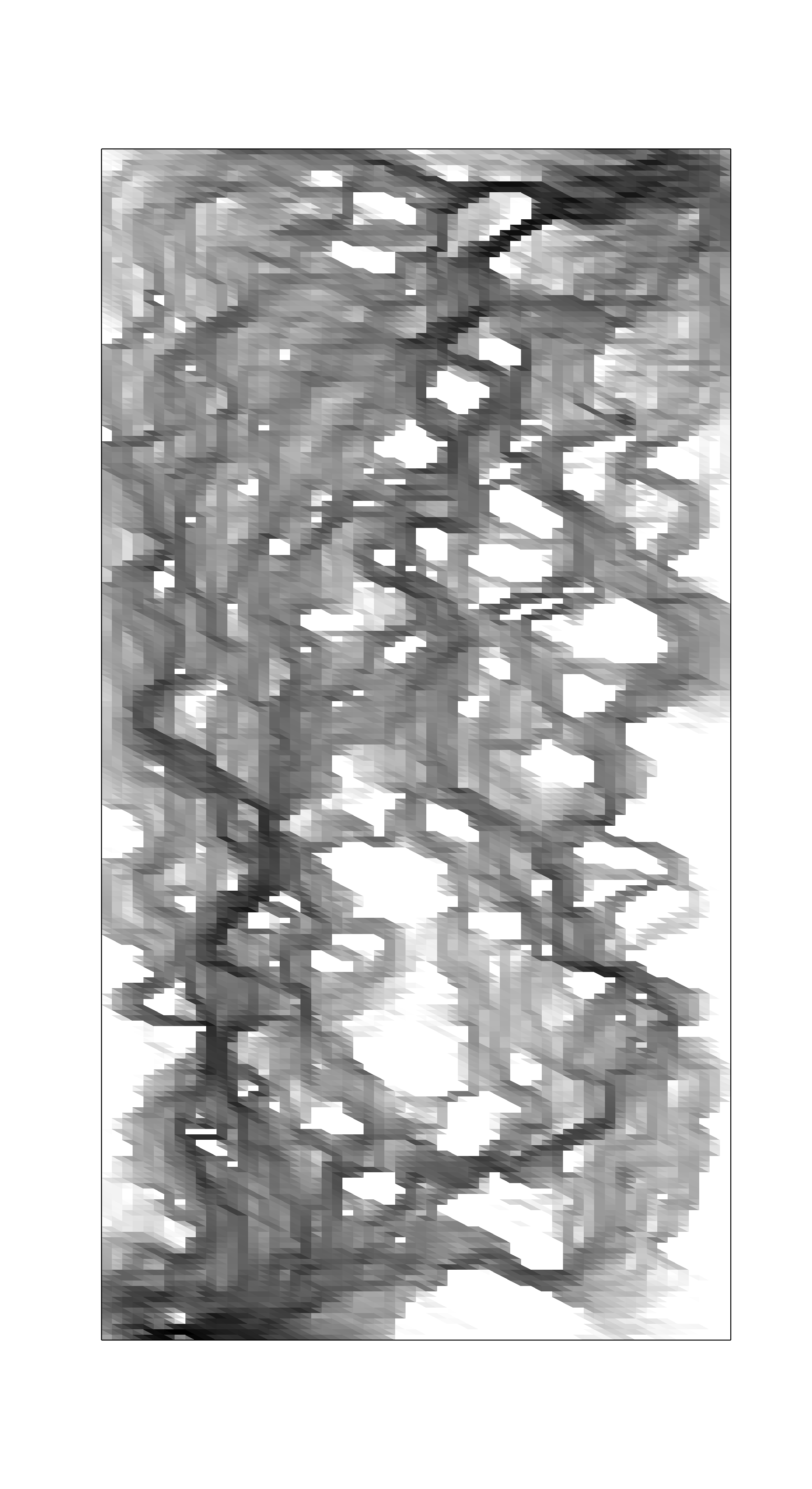}
    \caption{$k = 1, \ell = -1$}
  \end{subfigure}
  \begin{subfigure}{.30\textwidth}
    \centering
    \includegraphics[width=\textwidth]{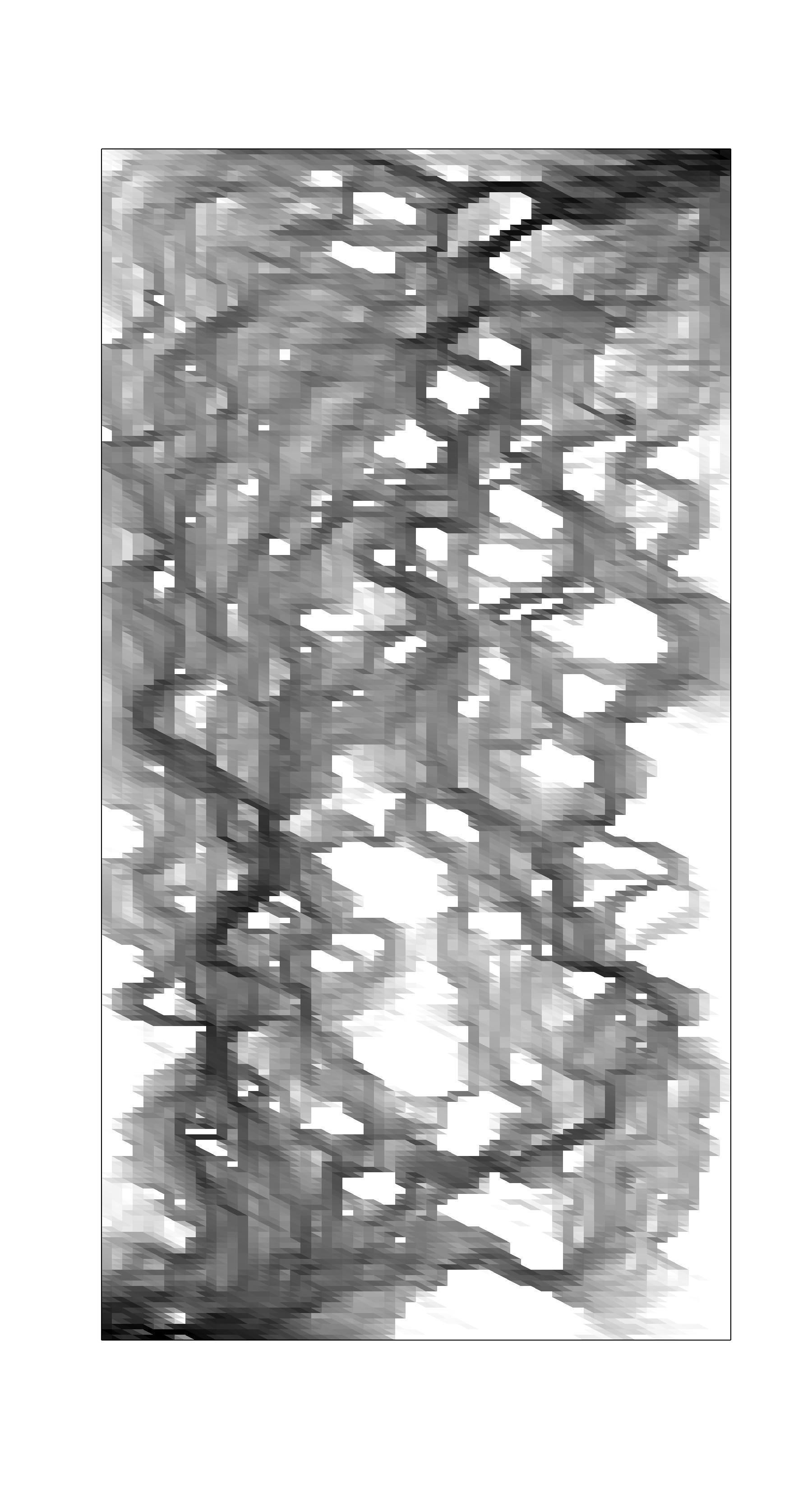}
    \caption{$k = 1, \ell = 0$}
  \end{subfigure}\\
  \begin{subfigure}{.30\textwidth}
    \centering
    \includegraphics[width=\textwidth]{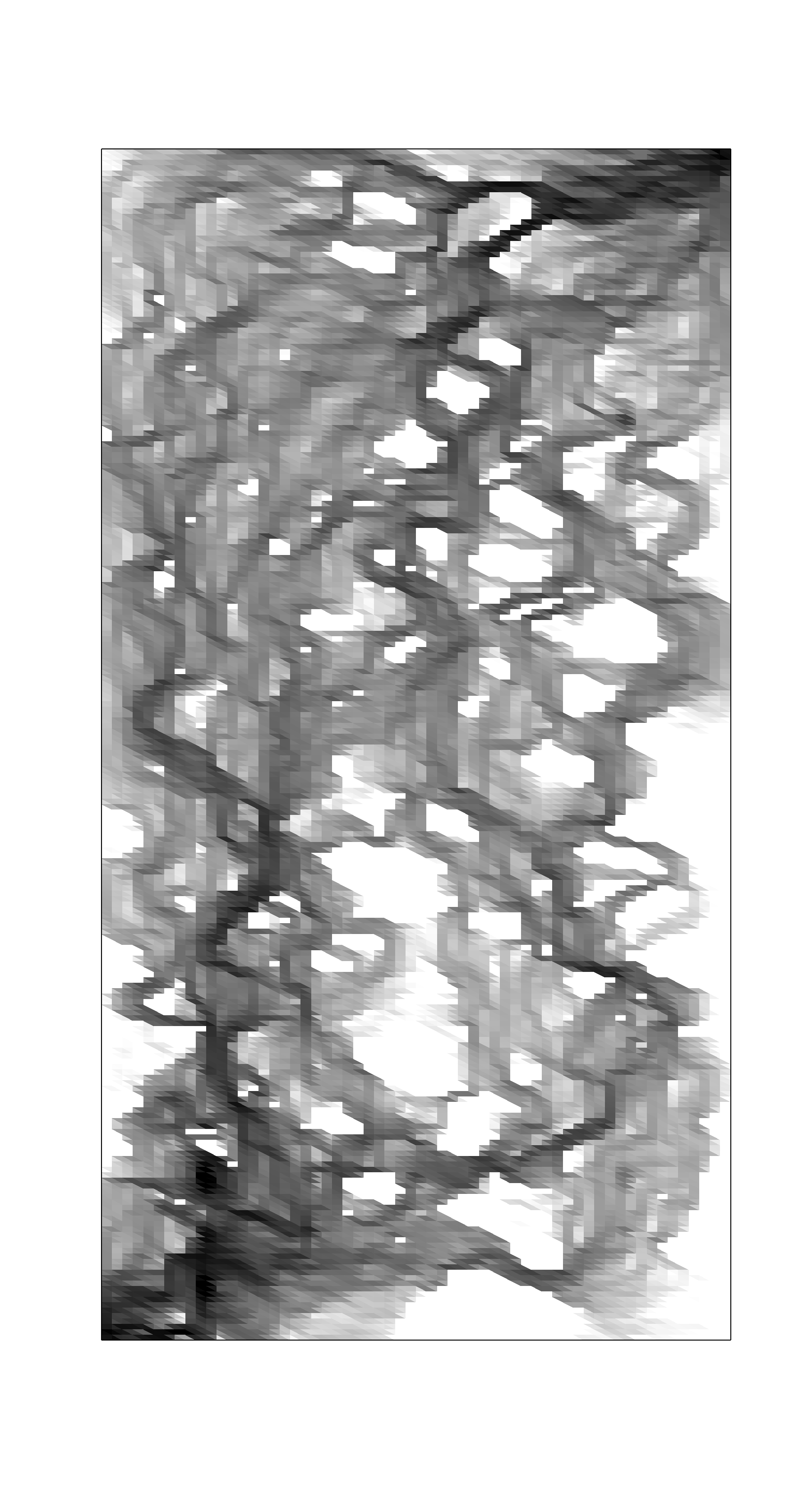}
    \caption{$k = 1, \ell = 1$}
  \end{subfigure}
  \begin{subfigure}{.30\textwidth}
    \centering
    \includegraphics[width=\textwidth]{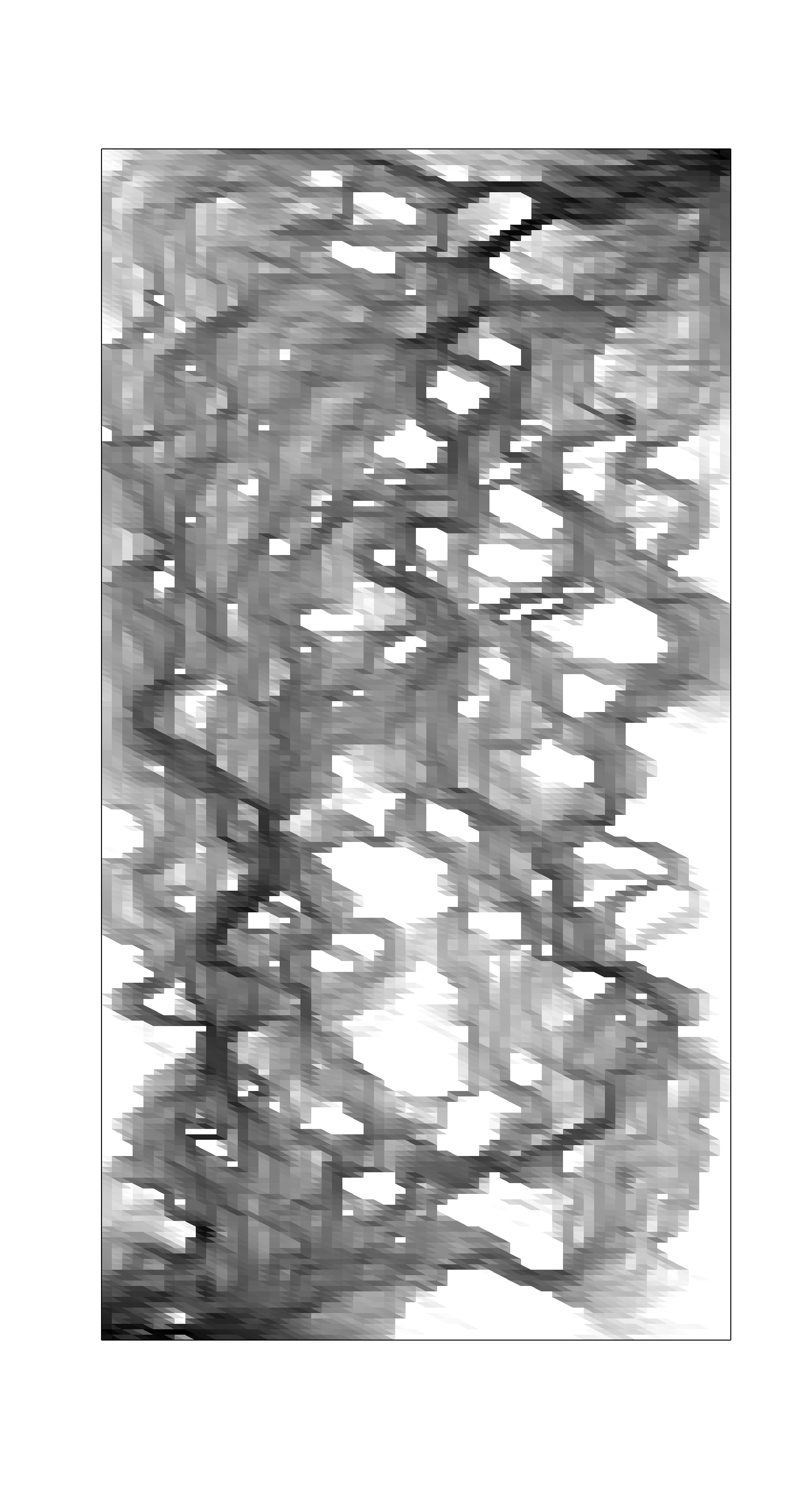}
    \caption{$k = 1, \ell = 2$}
  \end{subfigure}
  \begin{subfigure}{.30\textwidth}
    \centering
    \includegraphics[width=\textwidth]{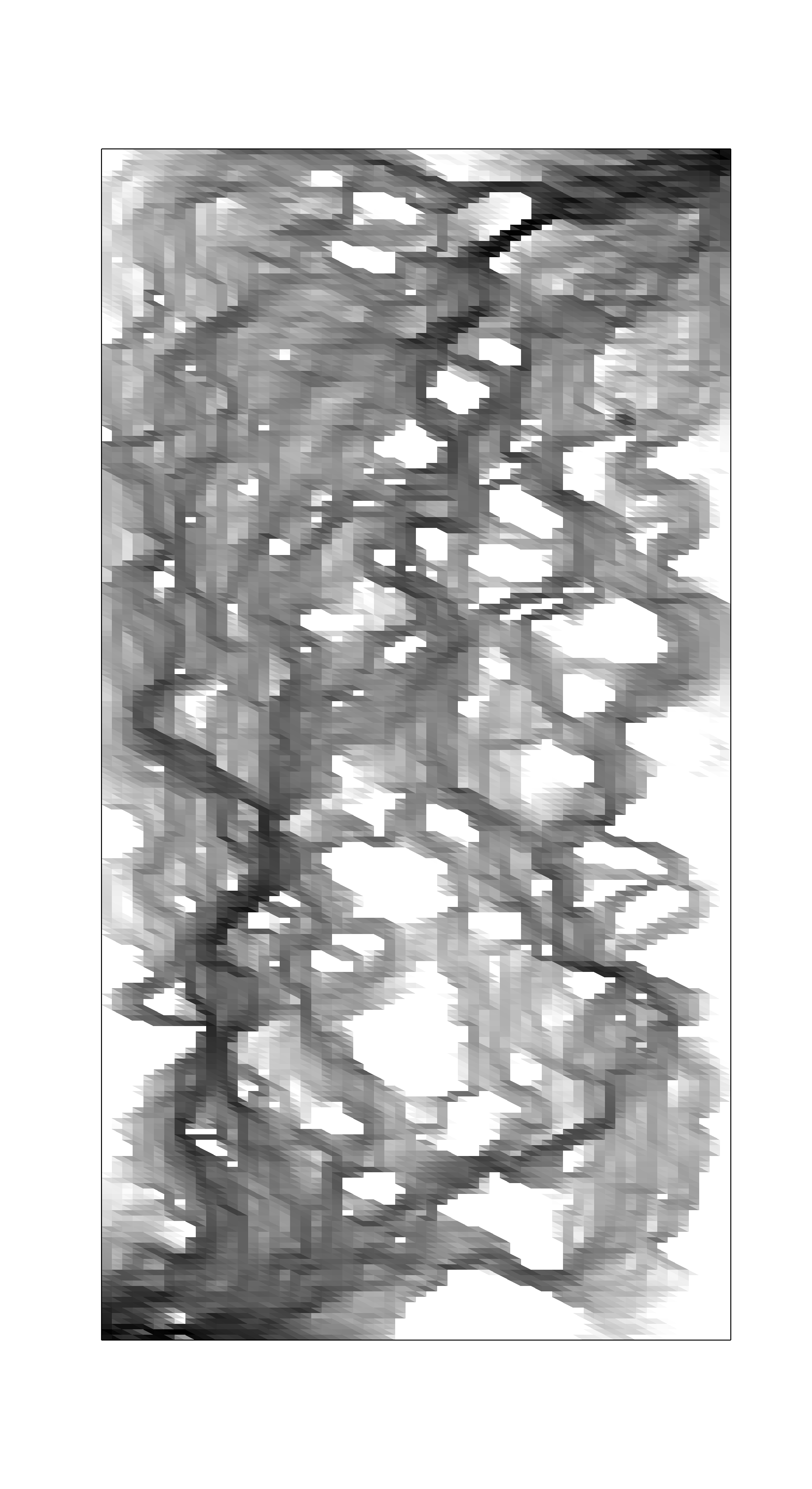}
    \caption{$k = 2, \ell = 3$}
  \end{subfigure}
  \caption{Flux solutions for the SPE10-85 problem. Figure (a) shows
    the reference flux solution and (b--f) show the multiscale flux
    solutions for $k = 1$ and $2$, and different source corrections
    ($\ell = -1$ means no source correction and $\ell = 0$ means
    ad-hoc error correction). The color maps to the magnitude of the
    flux at the midpoint of the triangular elements. The colors map
    from $10^{-5}$ (white) to $10^{-2}$ (black) and is saturated at
    white and black for lower and higher values, respectively.}
  \label{fig:speresults}
\end{figure}

\section*{Acknowledgement}
We gratefully acknowledge the anonymous reviewers for their careful
reading and insightful suggestions that improved the manuscript.

\end{document}